\documentclass[reqno,11pt]{amsart}

\usepackage{amsmath,amssymb}
\usepackage{graphicx}
\usepackage[english]{babel}
\usepackage[usenames]{color}
\usepackage{dsfont}
\usepackage{subfigure}
\usepackage{enumerate}

\usepackage{ifpdf}
\ifpdf
\usepackage{hyperref}
\else
\usepackage[hypertex]{hyperref}
\hypersetup{backref, colorlinks=true, bookmarks}
\fi

\newtheorem{theorem}{Theorem}[section]
\newtheorem{lemma}[theorem]{Lemma}

\newtheorem{definition}[theorem]{Definition}
\newtheorem{corollary}[theorem]{Corollary}
\theoremstyle{remark}
\newtheorem{remark}[theorem]{Remark}

\numberwithin{equation}{section}

\newcommand{\R}{\mathbb{R}}
\newcommand{\C}{\mathbb{C}}

\newcommand{\N}{\mathbb{N}}
\newcommand{\E}{\mathbb{E}}

\DeclareMathOperator{\erf}{erf}

\newcommand{\ET}[1]{\mathbb{E}_{\Theta}\left[#1\right]}

\newcommand{\ETt}[1]{\mathbb{E}_{\Theta}^t\left[#1\right]}
\newcommand{\PTt}[1]{\mathbb{P}_{\Theta}^t\left[#1\right]}
\newcommand{\VTt}[1]{\mathbb{V}_{\Theta}^t\left[#1\right]}

\newcommand{\En}[1]{\mathbb{E}_{\Theta}^n\left[#1\right]}
\newcommand{\Pn}[1]{\mathbb{P}_{\Theta}^n\left[#1\right]}

\newcommand{\Et}[1]{\mathbb{E}_{\Theta}^t\left[#1\right]}
\newcommand{\Pt}[1]{\mathbb{P}_{\Theta}^t\left[#1\right]}

\newcommand{\slan}{\sum_{\la \vdash n}\frac{1}{z_\la}}
\newcommand{\sla}{\sum_{\la}\frac{1}{z_\la}}

\newcommand{\la}{\lambda}

\newcommand{\nth}[1]{[t^n]\left[ #1 \right]}
\newcommand{\set}[1]{\left\{#1\right\}}
\DeclareMathOperator{\one}{\mathds{1}}

\renewcommand{\Re}{\mathrm{Re}}

\newcommand{\Sn}{\mathfrak{S}_n}

\DeclareMathOperator{\Li}{Li}

\DeclareMathOperator{\lcm}{lcm}

\begin{document}
\title[The Erd\H{o}s-Tur\'an law for weighted random permutations]{Total variation distance and the Erd\H{o}s-Tur\'an law for random permutations with polynomially growing cycle weights.}
\date{\today}

\author[J. Storm]{Julia Storm}
\address{Universit\"at Z\"urich\\Institut f\"ur Mathematik\\  Winterthurerstrasse 190\\ 8057-Z\"urich,
Switzerland} \email{julia.storm@math.uzh.ch}

\author[D. Zeindler]{Dirk Zeindler}
\address{Lancaster University\\ Mathematics and Statistics \\ Fylde College \\Bailrigg\\Lancaster\\United Kingdom\\LA1 4YF} \email{d.zeindler@lancaster.ac.uk}

\begin{abstract}
We study the model of random permutations of $n$ objects with polynomially growing cycle weights, which was recently considered by Ercolani and Ueltschi, among others. 
Using saddle-point analysis, we prove that the total variation distance between the process which counts the cycles of size $1, 2, ..., b$ and a process $(Z_1, Z_2, ..., Z_b)$ of independent Poisson random variables converges to $0$ if and only if $b=o(\ell)$ where $\ell$ denotes the length of a typical cycle in this model.
By means of this result, we prove a central limit theorem for the order of a permutation and thus extend the Erd\H{o}s-Tur\'an Law to this measure. Furthermore, we prove a Brownian motion limit theorem for the small cycles.
\end{abstract}

\maketitle

\tableofcontents

\section{Introduction} \label{section:introduction}

Denote by $\Sn$ the permutation group on $n$ objects. For a permutation $\sigma \in \Sn$  the order $O_n = O_n(\sigma)$ is the smallest integer $k$ such that the $k$-fold application of $\sigma$ to itself gives the identity. Landau \cite{La09} proved in 1909 that the maximum of the order of all $\sigma \in \Sn$ satisfies, for $n \rightarrow \infty$, the asymptotic
\begin{align}\label{eq:intro_asymp_Landau}
 \max_{\sigma \in \Sn}(\log (O_n)) \sim \sqrt{n \log (n)}.
\end{align}
 
 On the other hand, $O_n(\sigma)$ can be computed as the least common multiple of the cycle length of $\sigma$. Thus, if $\sigma$ is a permutation that consists of only one cycle of length $n$, then $\log(O_n(\sigma)) = \log(n)$, and $(n-1)!$ of all $n!$ possible permutations share this property. 
 Considering these two extremal types of behavior,  a famous result of Erd\H{o}s and Tur\'an \cite{ErTu65} seems even more remarkable: they showed in 1965 that, choosing $\sigma$ with respect to the uniform measure, a Normal limit law
\begin{align}\label{thm:intro_clt_order}
 \frac{\log (O_n) - \frac{1}{2}\log^2 (n)}{\sqrt{\frac{1}{3}\log^3 (n)}} \overset{d}{\longrightarrow} \mathcal{N}(0,1)
\end{align}
is satisfied as $n \rightarrow \infty$.

Several authors gave probabilistic proofs of this result, among them those of Best \cite{Be70} (1970), DeLaurentis and Pittel \cite{DePi85} (1985), whose proof is based on a functional central limit theorem for the cycle counts, and Arratia and Tavar\'e \cite{AT92b} (1992), who use the Feller coupling. This result was also extended to the Ewens measure and to A-permutations, see for instance \cite{AT92b} and \cite{Ya10a}.

In this paper we extend the Erd\H{o}s-Tur\'an Law to random permutations chosen according to a generalized weighted measure with polynomially growing cycle weights, see Theorem~\ref{thm:algrowth_CLT_logO_n}. One of our motivations is to find weights such that the order of a typical permutation with respect to this measure comes close to the maximum as in Landau's result.

To define the generalized weighted measure, denote by $C_m=C_m(\sigma)$ the number of cycles of length $m$ in the decomposition of the permutation $\sigma$ as a product of disjoint cycles. The functions $C_1$, $C_2$,\,\dots are random variables on $\Sn$ and we will call them cycle counts.  

\begin{definition}
\label{def:measure}
Let  $\Theta = \left(\theta_m  \right)_{m\geq1}$ be given with $\theta_m\geq0$ for every $m\geq 1$. We then define for $\sigma\in \Sn$
\begin{align}\label{eq:def:measure}
  \Pn{\sigma}
   :=
   \frac{1}{h_n n!} \prod_{m=1}^{n} \theta_{m}^{C_m} 
\end{align}
with $h_n = h_n(\Theta)$ a normalization constant and $h_0:=1$.
If $n$ is clear from the context, we will just write $\mathbb{P}_\Theta$ instead of $\mathbb{P}_\Theta^n$ .
\end{definition}

Notice that special cases of this measure are the uniform measure ($\theta_m \equiv 1$) and the Ewens measure ($\theta_m \equiv \theta$). 

These families of probability measures with weights depending on the length of the cycle recently appeared in particular in the work of Betz, Ueltschi and Velenik  \cite{BeUeVe11} and Erconali and Ueltschi \cite{ErUe11}. They studied a model of the quantum gas in statistical mechanics and the parameters $\theta_m$ may depend on the density, the temperature or the particle interaction and thus their structure might be complicated.

Subsequently, several properties of permutations have been studied under this measure by many authors and for different classes of parameters, see for instance \cite{HuNaNiZe11,Ma12a,Ma11,MaNiZe11,NiStZe13,NiZe11}. 
In this paper, the large $n$ statistics of $\log(O_n)$ are considered for polynomially growing parameters $\Theta = \left(\theta_m  \right)_{m\geq1}$ with $\theta_m = m^{\gamma}$, $\gamma >0$. Only few results are known for these parameters. Ercolani and Ueltschi \cite{ErUe11} show that under this measure, a typical cycle has length of order $n^{\frac{1}{1+\gamma}}$ and that the total number of cycles has order $n^{\frac{\gamma}{1+\gamma}}$. They also prove that the component process converges in distribution to mutually \textit{independent} Poisson random variables $Z_m$:
\begin{align}\label{eq:intro_convergence_fixed_b}
(C_1^n, C_2^n, ...) \overset{d}{\longrightarrow} (Z_1, Z_2, ...), \quad \quad \text{as } n \rightarrow \infty.
\end{align}
For many purposes this convergence is not strong enough, since it only involves the convergence of the vectors $(C_1^n, C_2^n, ..., C_b^n)$ for \textit{fixed} $b$. However, many natural properties of the component process jointly depend on all components, including the large ones, even though their contribution is less relevant. Thus, estimates are needed where $b$ and $n$ grow simultaneously.
The quality of the approximation can conveniently be described in terms of the total variation distance. For all $1 \leq b \leq n$ denote by $d_b(n)$ the total variation distance
\begin{align}\label{eq:intro_tot_var_dist_d_bn}
d_b(n) := d_{TV}\big(\mathcal{L}(C_1^n, C_2^n, ..., C_b^n), \mathcal{L}(Z_1, Z_2, ..., Z_b)\big).
\end{align}
For the uniform measure, where the $Z_m$ are independent Poisson random variables with mean $1/m$, it  was proved in 1990 by Barbour \cite{Ba90} that $d_b(n) \leq 2b/n$. This bound may be improved significantly. In 1992 Arratia and Tavar\'e \cite{AT92b} showed that 
\begin{align}\label{eq:intro_d_bn_for_Ewens}
d_b(n) \rightarrow 0 \quad \text{ if and only if } \quad b = o(n).
\end{align}
In particular, if $b = o(n)$, then  $d_b(n) \rightarrow 0$ superexponentially fast relative to $n/b$. The extension of these results to the Ewens measure is straightforward (here each $Z_m$ has mean $\vartheta/m$), but superexponential decay of $d_b(n)$ is only attained for $\vartheta =1$. For parameters $\vartheta \neq 1$ we have $d_b(n) = O(b/n)$, see Arratia et al. \cite[Theorem 6]{ArBaTa92}. For the uniform and the Ewens measure, the Feller coupling is used to study $d_b(n)$. When considering random permutations with respect to the weighted measure $\mathbb{P}_{\Theta}$, the Feller coupling is not available because of a lack of compatibility between the dimensions. Another approach is needed and it will turn out that for $\theta_m =m^\gamma$ the saddle point method is the right one to choose. We will prove in Section~\ref{section:algrowth_total_var_distance} that for appropriately chosen Poisson random variables $Z_m$ the following holds:
\begin{theorem}\label{thm:algrowth_total_variation}
Let $d_b(n)$ be defined as in \eqref{eq:intro_tot_var_dist_d_bn} and assume $\theta_m = m^{\gamma}, \gamma >0$. Then, as $n \rightarrow \infty$,
\begin{align}\label{eq:intro_d_bn_for_algrowth}
 d_b(n) \rightarrow 0 \quad \text{if and only if} \quad  b = o(n^{\frac{1}{1+\gamma}}).
\end{align}
Furthermore, if $b = o(n^{\frac{1}{1+\gamma}})$, then $d_b(n) = O\big(b^{2+\gamma}n^{-\frac{2+\gamma}{1+\gamma}} + b^{-\frac{\gamma}{6}} + n^{-\frac{\gamma}{1+\gamma}} \big)$. 
\end{theorem}

For the Ewens measure several applications demonstrating the power of \eqref{eq:intro_d_bn_for_Ewens} are available. Estimates like these unify and simplify proofs of limit theorems for a variety of functionals of the cycle counts, such as a Brownian motion limit theorem for cycle counts and the Erd\H{o}s-Tur\'an Law for the order of a permutation (see \eqref{thm:intro_clt_order}), among others; see \cite{ArTa92} for a detailed account. The basic strategy is as follows. 
First, choose an appropriate $b \rightarrow \infty$ and show that the contribution of the cycles of size bigger than $b$ is negligible. Second, approximate the distribution of the cycles of size at most $b$ by the independent limiting process, the error being controlled by the bound on the total variation distance.

Comparing \eqref{eq:intro_d_bn_for_Ewens} and \eqref{eq:intro_d_bn_for_algrowth} we notice that for polynomially growing parameters, the cycle counts exhibit a more dependent structure. 
An intuitive explanation is the following. In the Ewens case, a typical cycle has length of order $n$, and the numbers of cycles of length $o(n)$ are asymptotically independent of each other. 
For polynomially growing parameters $\theta_m = m^{\gamma}$, a typical cycle has length of order $n^{\frac{1}{1+\gamma}}$ (see \cite[Theorem 5.1]{ErUe11}), 
providing an intuitive justification for the bound on $b$ in \eqref{eq:intro_d_bn_for_algrowth}. 

The condition $b = o(n^{\frac{1}{1+\gamma}})$ in Theorem~\ref{thm:algrowth_total_variation} is much more restrictive than the condition $b=o(n)$ for the Ewens measure. 
Thus, the study of random variables involving almost all cycle counts $C_m$ is more difficult for weights $\theta_m = m^{\gamma}, \gamma >0$.
The reason is that in many cases the cycles with length longer than $n^{\frac{1}{1+\gamma}}$ have a non-negligible contribution 
(see also Remark~\ref{remark_bound_total_var_distance_too_small}). However, in Section~\ref{subsection:algrowth_erdoes_turan} we will show that \eqref{eq:intro_d_bn_for_algrowth} is useful to prove an analogue of the Erd\H{o}s-Tur\'an Law \eqref{thm:intro_clt_order} for our setting. We will prove
\begin{theorem}\label{thm:algrowth_CLT_logO_n}
Assume $\theta_m = m^{\gamma}$ with $0< \gamma < 1$. Then, as $n \rightarrow \infty$, 
\begin{align*}
\frac{\log (O_n) - G(n)}
{\sqrt{ F(n)}} 
\overset{d}{\longrightarrow} \mathcal{N}(0,1)
\end{align*}
where $\mathcal{N}(0,1)$ denotes the standard Gaussian distribution and 
\begin{align*}
F(n)&= \frac{K(\gamma)}{(1+\gamma)^3} \, n^{\frac{\gamma}{1+\gamma}} \log^2 (n),\\
G(n)&= \frac{K(\gamma)}{1+\gamma} \, n^{\frac{\gamma}{1+\gamma}} \log(n) + n^{\frac{\gamma}{1+\gamma}} H(n) \quad \text{with}\\
H(n) &= K(\gamma) \bigg(\frac{\Gamma'(\gamma)}{\Gamma(\gamma)}- \frac{\log(\Gamma(1+\gamma))}{1+\gamma} \bigg) ;
\end{align*} 
here $\Gamma'$ denotes the derivative of the gamma function and
$$K(\gamma) = \Gamma(\gamma) \Gamma(1+\gamma)^{-\frac{\gamma}{1+\gamma}} .$$
\end{theorem}
In particular, notice that for $0 < \gamma < 1$ there exists constants $c,C>0$ such that
\begin{align*}
c \,n^{\frac{\gamma}{1+\gamma}} \log(n) 
\leq 
\ET{\log(O_n)} 
\leq
C\, n^{\frac{\gamma}{1+\gamma}} \log(n).
\end{align*}
Thus, for our choice of parameters the mean of $\log(O_n)$ is in fact very close to Landau's result \eqref{eq:intro_asymp_Landau}.
Unfortunately, our approach does not work for $\gamma\geq 1$ and thus the behavior in this situation is currently unknown.

Furthermore, though the bound in \eqref{eq:intro_d_bn_for_algrowth} is too small to investigate the whole cycle count process via the independent Poisson process, we will present in Section~\ref{subsection:FCLT} how \eqref{eq:intro_d_bn_for_algrowth} may be used to study the small components by proving a functional version of the Erd\H{o}s-Tur\'an Law.
For $x > 0$ define $x^* := \lfloor x \, n^{\frac{\gamma}{1+\gamma}} \rfloor$ and 
\begin{align}\label{eq:algrowth_brownian_functional}
B_n(x) := \frac{\log(O_{x^*}) -\frac{1}{1+\gamma} \, x^{\gamma} \log(n) \, n^{\frac{\gamma^2}{1+\gamma}}}{\sqrt{\frac{\gamma}{(1+\gamma)^2} \log^2(n) \, n^{\frac{\gamma^2}{1+\gamma}}}},
\end{align}
where $O_{x^*}(\sigma):= \lcm\{m\leq x^*; C_m>0\}$. We will prove the following
\begin{theorem}\label{thm:fclt}
Assume $\theta_m = m^{\gamma}$ with $0 <\gamma < 1$, take $B_n(x)$ as in \eqref{eq:algrowth_brownian_functional} and denote by $\mathcal{W}$ a standard Brownian motion. Then, as $n \rightarrow \infty$ and for $x > 0$, $B_n(x)$ converges weakly to $\mathcal{W}(x^{\gamma})$ .
\end{theorem}

Apart from the behavior of the small cycle counts and the Erd\H{o}s-Tur\'an Law, it might be interesting to study further properties of $\log(O_n)$. We refer the reader to \cite{StZe14b} for more results on the order of random permutations with polynomially growing cycle weights, such as large deviation estimates and local limit theorems, that are new even for the Ewens measure.

\vskip 40pt 

\section{Generalities}
\label{sec_comb_andgen_of_Sn}

We present in this section some facts about the symmetric group $\Sn$, partitions and generating functions. In particular, two useful lemmas, which identify averages over $\Sn$ with generating functions, are recalled. 
We give only a short overview and refer to \cite{Ap84}, \cite{ABT02} and \cite{Mac95} for more details. At the end of this section we present some basic facts about the saddle-point method, which is the main tool we will apply to get our results.

\vskip 15pt 

\subsection{The symmetric group}\label{sec21}

All probability measures and functions considered in this paper are invariant under conjugation and 
it is well known that the conjugation classes of $\Sn$ can be parametrized with partitions of $n$.
This can be seen as follows: Let $\sigma\in \Sn$ be an arbitrary permutation and write $\sigma = \sigma_1\cdots \sigma_\ell$ with $\sigma_i$ disjoint cycles of length $\la_i$.
Since disjoint cycles commute, we can assume that $\la_1\geq\la_2\geq\cdots \geq\la_\ell$.
We call the partition $\la=(\la_1,\la_2,\cdots,\la_\ell)$ the \emph{cycle-type} of $\sigma$ and $\ell = \ell(\la)$ its \emph{length}. Then two elements $\sigma,\tau\in \Sn$ are conjugate if and only if
$\sigma$ and $\tau$ have the same cycle-type. Further details can be found for instance in \cite{Mac95}.
For $\sigma\in \Sn$ with cycle-type $\la$ we define $C_m$, the number of cycles of size $m$, that is 
\begin{align}
\label{eq_def_Cm_lambda}
C_m :=  \#\set{i ;\la_i = m} .
\end{align}
Recall that $u$ is a class function when it satisfies $u(\sigma)= u(\tau^{-1}\sigma\tau)$ for all $\sigma,\tau\in\Sn$. 
It will turn out that all expectations of interest have the form $\frac{1}{n!} \sum_{\sigma\in \Sn} u(\sigma)$ for a certain class function $u$.
Since $u$ is constant on conjugacy classes, it is more natural to sum over all conjugacy classes. This is the subject of the following lemma.

\begin{lemma}
\label{lem:size_of_conj_classes}
Let  $u: \Sn \to \C$ be a class function, $C_m$ be as in \eqref{eq_def_Cm_lambda} and $\mathcal{C}_\la$ the conjugacy class corresponding to the partition $\la$. We then have
%
%
\begin{align*}
  \frac{1}{n!} \sum_{\sigma\in \Sn} u(\sigma)
  =
  \slan u(\mathcal{C}_\la) 
\end{align*}
with $z_\la:=\prod_{m=1}^{n} m^{C_m}C_m!$ and $\sum_{\la \vdash n}$ the sum over all partitions of $n$.
\end{lemma}

\vskip 15pt 

\subsection{Generating functions}\label{subsection:generating_functions}

Given a sequence $(a_{n})_{n\in\N}$ of numbers, one can encode important information about this sequence into a formal power series called the generating series.
\begin{definition}
\label{def_gneranting_function}
Let $(a_n)_{n\in\N}$ be a sequence of complex numbers. We then define
the generating function of $(a_n)_{n\in\N}$ as the formal power series
    \begin{align*}
      G(t) = G(a_n,t) = \sum_{n=0}^\infty a_n t^n.
    \end{align*}
%
We define $\nth{G(t)}$ to be the coefficient of $t^n$ of $G(t)$, that is $$\nth{G(t)} := a_n.$$ 
\end{definition}
The reason why generating functions are powerful is the possibility of recognizing them
without knowing the coefficients $a_n$ explicitly. In this case one can try to use tools from analysis to extract information about $a_{n}$, for large $n$, from the generating function. 

The following lemma goes back to P\'olya and is sometimes called \textit{cycle index theorem}. It links generating functions and averages over $\Sn$.

\begin{lemma}
\label{lem:cycle_index_theorem}
Let $(a_m)_{m\in\N}$ be a sequence of complex numbers. Then
\begin{align*}
\sum_{n=0}^\infty \frac{t^n}{n!} \sum_{\sigma\in\Sn} \prod_{m=1}^{\infty} a_m^{C_m}
= \sla \left(\prod_{m=1}^{\infty} (a_m t^m)^{C_m}\right) 
=
\exp\left(\sum_{m=1}^{\infty}\frac{a_m}{m} t^m\right)
\end{align*}
with the same $z_\la$ as in Lemma~\ref{lem:size_of_conj_classes}.
If one of the sums above is absolutely convergent then so are the others.
\end{lemma}
\begin{proof}
The proof can be found in \cite{Mac95} or can be directly verified using the definitions of $z_\la$ and the exponential function.
The last statement follows from the dominated convergence theorem.
\end{proof}
The previous lemma implies
\begin{corollary}\label{cor:g_theta_relation_hn}
Define the generating function
\begin{align*}
g_\Theta(t):= \sum_{m=1}^\infty \frac{\theta_m}{m}t^m,
\end{align*}
and let $h_n$ be as in Definition~\ref{def:measure}. Then the following holds
\begin{align}
\label{eq:generating_hn}
\sum_{n=0}^\infty h_n t^n = \exp(g_\Theta(t)).
\end{align}
\end{corollary}
\begin{proof}
This follows immediately from the definition of $h_n$ in \eqref{eq:def:measure} together with Lemma~\ref{lem:cycle_index_theorem}.
\end{proof}

The generating function \eqref{eq:generating_hn} yields expressions for the factorial moments of the cycle counts. 
\begin{lemma}
\label{lem:Cm_expectation}
We have for all $m,k\in\N$,
\begin{align*}
\ET{(C_m)_k} = \left(\frac{\theta_m}{m}\right)^k \frac{h_{n-mk}}{h_n},
\end{align*}
where $(c)_k:=c(c-1)\cdots(c-k+1)$ denotes the Pochhammer symbol. Furthermore,  for $m_1\neq m_2$,
\begin{align*}
\ET{C_{m_1}C_{m_2}}= \frac{\theta_{m_1}}{m_1}\frac{\theta_{m_2}}{m_2}\frac{h_{n-m_1-m_2}}{h_n}.
\end{align*}
\end{lemma}

\begin{proof}
Recall Lemma~\ref{lem:cycle_index_theorem} and set  $a_m=\theta_m$, then differentiate the sum $k$ times with respect to $\theta_m$ and obtain
\begin{align}
\label{eq:Cm_fact_moments_2}
\sum_{n=0}^\infty h_n \ET{(C_m)_k}t^n  
= 
\sum_{n=0}^\infty \frac{t^n}{n!} \sum_{\sigma\in\Sn}  (C_m)_k \prod_{m=1}^{\infty} \theta_m^{C_m}
=
\left(\frac{\theta_m}{m}t^m\right)^k \exp(g_\Theta(t)).
\end{align}
Taking $\nth{.}$ on the left- and right-hand side completes the proof of the first assertion in Lemma~\ref{lem:Cm_expectation}. The proof of the second assertion is similar and we thus omit it.
\end{proof}

\begin{remark}\label{remark:C_b-for-fixed-b-converge-to-indep-Poisson}
It is now easy to see that under the mild condition $\frac{h_{n-1}}{h_n}\to r$ the convergence \eqref{eq:intro_convergence_fixed_b} holds with $\E_{\Theta}[Z_m]= \frac{\theta_m}{m}r^m$; see for instance \cite{ErUe11}.
\end{remark}
 
Typically, Lemma~\ref{lem:Cm_expectation} is used in cases where one can express the quantity of interest in terms of the factorial moments of $C_m$. 
However, in our case it proves simpler to take a different approach, which
 was in particular applied by Hansen \cite{Ha90}. 

Assume for $t>0$ that $G_\Theta(t) := \exp \bigl( g_\Theta(t)\bigr) < \infty$
with $g_\Theta(t)$ as in Corollary~\ref{cor:g_theta_relation_hn}. Then set 
\begin{align*}
  \Omega_t:= \dot{\bigcup}_{n\in\N} \Sn                                                         
\end{align*}
and define for $\sigma\in\Sn$
\begin{align*}
\Pt{\sigma}:=\frac{1}{G_\Theta(t)}\frac{t^{n}}{n!} \prod_{m=1}^{n} \theta_{m}^{C_m}.
\end{align*}
Lemma~\ref{lem:cycle_index_theorem} shows that $\mathbb{P}_{\Theta}^t$ defines a probability measure on $\Omega_t$. 
Furthermore, the $C_m$ are independent and Poisson distributed with $\Et{C_m} = \frac{\theta_m}{m}t^m$.
This follows easily with a calculation similar to the proof of Lemma~\ref{lem:Cm_expectation}.
The following conditioning relation holds:
\begin{align}\label{eq:intro_randomization_Pt}
 \Pt{\,\cdot\,|\,\Sn} = \Pn{\,\cdot\,}.
\end{align}
We also have
\begin{align*}
\Pt{\Sn}
=
t^n h_n \exp\bigl(- g_\Theta(t)\bigr),
\end{align*}
which follows immediately from the definition of $h_n$ in \eqref{eq:def:measure}.
Then the law of total probability yields
\begin{lemma}\label{lem:cycle_index2}
Let $t>0$ be given so that $G_\Theta(t)<0$. Suppose that $\Psi: \Omega_t\to\C$ is a random variable with $\Et{|\Psi|}< \infty$ and that $\Psi$ only depends on the cycle counts, i.e. $\Psi = \Psi(C_1,C_2,\dots)$. 
We then have with $\Psi_n:=\Psi\Big|_{\Sn}$
\begin{align*}
\exp\bigl(g_\Theta(t)\bigr) \Et{\Psi} = \sum_{n=1}^{\infty} h_n \En{\Psi_n} t^n + \Psi(0).
\end{align*} 
\end{lemma}
The previous equation is stated only for fixed $t$, 
but if both sides are complex analytic functions in $t$, then the equation is also valid as formal power series.
If one chooses for instance $\Psi= (C_m)_k$, one gets $\Et{\Psi}= (\theta_m/m)^k t^{mk}$ and thus obtains \eqref{eq:Cm_fact_moments_2}.

In Section~\ref{section:algrowth_total_var_distance} we will compare the distribution of the cycle counts $C_m$ 
under $\mathbb{P}^n_\Theta$ and under $\mathbb{P}^t_\Theta$. To avoid confusion, we will write $Z_m$ instead of $C_m$ if we consider the measure $\mathbb{P}^t_\Theta$. Then the $Z_m$ are independent Poisson random variables with mean $\frac{\theta_m}{m}t^m$. Notice that \eqref{eq:intro_randomization_Pt} implies the so-called \textit{Conditioning Relation}
\begin{align}\label{eq:intro_conditioning relation}
\mathcal{L}\big((C_1, ..., C_n) \big) = \mathcal{L}\Big((Z_1, ..., Z_n)|\sum_{k=1}^n k Z_k =n \Big).
\end{align}
This important relation is necessary for the proof of Theorem~\ref{thm:algrowth_total_variation}.

To demonstrate how to further use the randomization method, let us compute the generating series of a functional of the cycle counts that we will need in Section~\ref{section:erdos-turan-law}. Define on $\Sn$
$$\log(Y_n) := \sum_{m=1}^n  C_m  \log(m).$$
We then have on $\Omega_t$ with the above convention
$$
\log(Y_n) = \sum_{m=1}^n Z_m \log(m).
$$
Since $Z_1, Z_2, ..., Z_n$ are independent Poisson random variables with respective parameters $\frac{\theta_m}{m} t^m$, we obtain
\begin{align*}
\Et{e^{s\log(Y_n)}}
=
\Et{e^{s \sum_{m=1}^n Z_m \log(m)}}
= \exp\Big(\sum_{m=1}^n \frac{\theta_m}{m}t^m(e^{s\log(m)}-1) \Big)
\end{align*}
and then Lemma~\ref{lem:cycle_index2} yields
\begin{align}\label{eq:generating_series_logYn_allgemein}
 \sum_{n=0}^{\infty} h_n \E_{\Theta}[\exp(s \log (Y_n))] t^n  
= \exp\left( \sum_{m=1}^{\infty} \frac{\theta_m}{m^{1 - s}} t^m  \right).
\end{align}
%

\vskip 15pt
\subsection{Saddle point analysis}\label{subsection:saddle-point-analysis}
The asymptotic behavior of random variables on the symmetric group $\Sn$ strongly depends on the analytic properties of $g_{\Theta}$ as defined in Corrolary~\ref{cor:g_theta_relation_hn}. Thus, the appropriate method for studying generating functions involving $g_{\Theta}$  depends on the parameters $\theta_m$.

For our choice $\theta_m = m^{\gamma}$, the function $g_{\Theta}$ belongs to the class of so-called log-admissible functions as defined in \cite[Definition 2.1]{MaNiZe11}. Thus, a suitable method to investigate the behavior of functionals we are interested in is the saddle-point method. 

Consider the generating series
\begin{align*}
\exp(g(t,s)) = \sum_{n=0}^{\infty}G_{n,s}t^n.
\end{align*}
If $g(t,s)$ is $\log$-admissible, then the asymptotics of the coefficients $G_{n,s}$ can be computed explicitly, see Lemma~\ref{lem:algrowth_expansion Gns} below.
\begin{definition}\label{def_log-admissible}
 Let $g(t)=\sum_{n\geq0} g_{n} t^n$ be given with radius of convergence $\rho>0$ and $g_n \geq0$ for all $n$. Then $g(t)$ is called $\log$-admissible if there exist functions $\alpha,\beta,\delta : [0,\rho) \rightarrow \R^+$ and $R: [0, \rho) \times (-\pi/2, \pi/2) \rightarrow \R^+$ with the following properties:
\begin{description}
  \item [Approximation] For all $|\phi| \leq \delta(r)$ the expansion
    \begin{align*}
    g(re^{i\phi}) = g(r) + i\phi \alpha(r) - \frac{\phi^2}{2} \beta(r) + R(r, \phi)
    \end{align*}
    holds, where $R(r, \phi) = o(\phi^3 \delta(r)^{-3})$.
  \item [Divergence] $\alpha(r) \rightarrow \infty$, $\beta(r) \rightarrow \infty$  and $\delta(r) \rightarrow 0$ as $r \rightarrow \rho$.
  \item [Width of convergence] We have $\epsilon \delta^2(r)\beta(r) - \log \beta(r) \rightarrow \infty$ for all $\epsilon > 0$ as $r \rightarrow \rho$.
  \item [Monotonicity] $\Re \big(g(re^{i\phi})\big) \leq \Re \big(g(r e^{\pm i \delta(r)})\big)$ holds for all $|\phi| > \delta(r)$.
 \end{description}
\end{definition}
In Section~\ref{section:erdos-turan-law} we will need to study functions $g$ with an additional dependence on a parameter $s$.
In this case we will use 
\begin{lemma}\label{lem:algrowth_expansion Gns}
Let $I\subset \R$ be an interval and suppose that $g(t,s)$ is a smooth function for $s\in I$ and $|t|\leq \rho$.
Suppose further that $g(t,s)$ is $\log$-admissible in $t$ for all $s\in I$ with associated functions $\alpha_s$, $\beta_s$. Let further $r_{xs}$ be the unique solution of $\alpha_s(r)=x$.
If the requirements of Definition~\ref{def_log-admissible} are fulfilled uniformly in $s$ for $s$ bounded, 
then, as $n \rightarrow \infty$, the following asymptotic expansion holds:
\begin{align*}
G_{n,s}=\frac{1}{\sqrt{2\pi}}(r_{ns})^{-n}\beta_s(r_{ns})^{-1/2} \exp(g(r_{ns}, s)) \big(1+o(1)\big)
\end{align*}
uniformly in $s$ for $s$ bounded.
\end{lemma}
The proof of Lemma~\ref{lem:algrowth_expansion Gns} is analogue to the proof of Proposition 2.2 in \cite{MaNiZe11}; one simply has to verify that all involved expression are uniform in $s$. This is straightforward and we thus omit the details.

\begin{remark}
\label{rem:saddle_point}
It is often difficult to find the exact solution of $\alpha_s(r)=x$, 
fortunately it is enough to find $r_{ns}$ with 
\begin{align}
\label{eq:saddle_solution}
 \alpha_s(r_{ns}) = n +o\left(\sqrt{\beta_s(r_{ns})}  \right), 
\end{align}
since then the contribution of the error term is negligible in the limit.
\end{remark}

Let us apply this method to study the asymptotic behavior of $h_n$ as defined in Definition~\ref{def:measure}. Recall Corollary~\ref{cor:g_theta_relation_hn}, which states that
$$h_n = [t^n] \exp(g_{\Theta}(t)).$$
We have to show that $g_{\Theta}$ is log-admissible. This will be proved in a more general way in Lemma~\ref{lem:algrowth_g(r,s)_is_log_adm} in Section~\ref{subsection:algrowth_preliminaries}. Then Lemma~\ref{lem:algrowth_expansion Gns} yields

\begin{corollary}
Let $g_{\Theta}$ be as in Corollary~\ref{cor:g_theta_relation_hn} with $\theta_m = m^{\gamma}$, $\gamma >0$. Then
\begin{align}
h_n 
= & \, \big(2\pi \Gamma(2+\gamma)\big)^{-\frac{1}{2}} \Big(\frac{\Gamma(1+\gamma)}{n} \Big)^{\frac{2+\gamma}{2(1+\gamma)}}  \times \nonumber\\
&\exp\bigg(\frac{1+\gamma}{\gamma} \,\Gamma(1+\gamma)^{\frac{1}{1+\gamma}} \, n^{\frac{\gamma}{1+\gamma}} + \zeta(1-\gamma) \bigg)\big(1+o(1)\big). \label{eq:intro_algrowth_h_n}
\end{align}
\end{corollary}
\begin{proof}
This is a special case of the proof of Theorem~\ref{thm:algrowth_mom-gen_logYn} in Section~\ref{subsection:algrowth_preliminaries}.
\end{proof}

\begin{remark} \label{rem:complete_asymptotic}	
We will need for the proof of the rate of convergence in Theorem~\ref{thm:algrowth_total_variation} a more precise asymptotic expansion for $G_{n,s}$ than the one in Lemma~\ref{lem:algrowth_expansion Gns}.
This can be obtained by taking into account more error terms in the $\phi-$expansion of $g(t,s)$ at $t=r$. Often one can indeed obtain a complete asymptotic expansion. 
The details are explained for instance in \cite[Chapter VIII]{FlSe09}.
For us this means that if 
\begin{align*}
 R(r, \phi) = c_n(r) \phi^3 + O\left(d_n(r) \phi^4\right)  
\end{align*}
then the $o(1)$ error-term in Theorem~\ref{thm:algrowth_total_variation} is
\begin{align*}
O\left(\frac{d_n(r_{ns})}{\beta_s(r_{ns})^{2}} +\frac{c_n(r_{ns})^2}{\beta_s(r_{ns})^{3}}  \right).
\end{align*}
Applying this to $h_n$ gives
\begin{align}
h_n 
= & \, \big(2\pi \Gamma(2+\gamma)\big)^{-\frac{1}{2}} \Big(\frac{\Gamma(1+\gamma)}{n} \Big)^{\frac{2+\gamma}{2(1+\gamma)}}  \times \nonumber\\
&\exp\bigg(\frac{1+\gamma}{\gamma} \,\Gamma(1+\gamma)^{\frac{1}{1+\gamma}} \, n^{\frac{\gamma}{1+\gamma}} + \zeta(1-\gamma) \bigg)\big(1+O(n^{-\frac{\gamma}{1+\gamma}})\big). 
\label{eq:intro_algrowth_h_n2}
\end{align}
\end{remark}

\vskip 15pt 
\subsection{Asymptotics}
We recall the asymptotic behavior of several functions that we will encounter frequently throughout the paper.
The upper incomplete gamma function is defined as
\begin{align*}
\Gamma(a,y) := \int_y^{\infty} x^{a-1}\,e^{-x}\,dx
\end{align*}
and satisfies
\begin{align} \label{eq:intro_asym_Gamma_0}
 \Gamma(a,y) &= \Gamma(a) - \frac{1}{a} y^{a} + \Sigma_{2}(a,y) 
 \quad \text{ as } y \rightarrow 0,
 \end{align}
with 
\begin{align}
\label{eq:def_sigma}
\Sigma_j(a,y) = \sum_{k=j}^{\infty}(-1)^k \frac{y^{k-1+a}}{(k-1)! (k-1+a)}
 \end{align}

 and 
\begin{align} \label{eq:intro_asym_Gamma_infty}
 \Gamma(a,y) &= e^{-y} y^{a-1} (1+O(1/y)) ,
 \quad \text{ as } y \rightarrow \infty.
\end{align}
Furthermore, the Error function is defined as 
\begin{align*}
\erf(x) := \frac{2}{\sqrt{\pi}}\int_{0}^x e^{-t^2}\, dt
\end{align*}
and satisfies
\begin{align}\label{eq:intro_erf_infty}
\erf(x) = 1 + O(x^{-1} e^{-x^2}) \quad \text { as } x \rightarrow \infty
\end{align}
and 
\begin{align}\label{eq:intro_erf_minus_infty}
\erf(x) = -1 + O(x^{-1} e^{-x^2}) \quad \text { as } x \rightarrow -\infty.
\end{align}
Recall that the polylogarithm $\Li_{a}$ with parameter $a$ is defined as
\begin{align*}
\Li_a(t):= \sum_{k=1}^{\infty} \frac{t^k}{k^a}.
\end{align*}
Its radius of convergence is $1$ and as $t \rightarrow 1$ it satisfies the following asymptotic for $a \notin \{1, 2, ... \}$ 
(see \cite[Theorem VI.7]{FlSe09}):
\begin{align}\label{eq:algrowth_asypm_polylog}
\Li_{a}(t)
\sim \Gamma(1-a)(-\log (t))^{a -1} + \sum_{j \geq 0} \frac{(-1)^j}{j!}\xi(a - j) (-\log (t))^j .
\end{align}
In particular, for $a < 1$, \eqref{eq:algrowth_asypm_polylog} implies for $t \rightarrow 1$
\begin{align}\label{eq:algrowth_asypm_polylog_2}
 \Li_{a}(t) 
=  \Gamma(1-a)(-\log (t))^{a-1} +\zeta(a) +O(t-1).
\end{align} 
Finally, recall the Euler-Maclaurin formula
\begin{align}\label{eq:intro_euler_maclaurin}
\sum_{j=1}^b f(j)
= \int_1^b f(x) dx + \int_1^b (x-\lfloor x \rfloor)f'(x) dx + f(b)(b-\lfloor b \rfloor)  .
\end{align}

\vskip 40pt

\section{Total variation distance}\label{section:algrowth_total_var_distance}

This section is devoted to the proof of Theorem~\ref{thm:algrowth_total_variation}. Recall the randomization method we considered at the end of Section~\ref{subsection:generating_functions}, and the independent Poisson random variables $Z_m$ with mean $\frac{\theta_m}{m}t^m$ we introduced there. Define 
\begin{align*}
T_{\ell k} := \sum_{m=\ell+1}^k m \, Z_m.
\end{align*}
and recall also that the Conditioning Relation \eqref{eq:intro_conditioning relation} holds. We rewrite it as
\begin{align}\label{eq:tot_var_dist_cond_rel}
\mathcal{L}\big((C^n_1, C^n_2, ..., C^n_n)\big) = \mathcal{L}\big((Z_1, Z_2, ..., Z_n) | T_{0n}=n\big).
\end{align}
Recall that we denote by $d_b(n)$ the total variation distance
\begin{align}\label{eq:tot_var_dist_d_bn}
d_b(n) = d_{TV}\big(\mathcal{L}(C_1^n, C_2^n, ..., C_b^n), \mathcal{L}(Z_1, Z_2, ..., Z_b) \big) .
\end{align}
We have to prove that
\begin{align*}
d_b(n) \rightarrow 0
\quad \text{ if and only if }\quad  b = o(n^{\frac{1}{1+\gamma}}).
\end{align*}
Given the Conditioning Relation \eqref{eq:tot_var_dist_cond_rel}, Lemma 1 in \cite{ArTa92} gives a formula that reduces the total variation distance of two vectors to the distance of two one-dimensional random variables:
\begin{align}\label{eq:tot_var_dist_d_bn_one_dim_rv}
d_b(n) = d_{TV}\big(\mathcal{L}(T_{0b}),\mathcal{L}(T_{0b}|T_{0n}=n)\big).
\end{align}
Then $d_b(n) \rightarrow 0 $ implies that conditioning on the event $\{T_{0n}=n \}$ does not change the distribution of $T_{0b}$ very much, which is indeed the case when $\{T_{0n}=n \}$ is relatively likely. Recall that for uniform random permutations \eqref{eq:intro_d_bn_for_Ewens} holds; in this setting, one can compute that $\PTt{T_{0n}=n}$ is approximately $n^{-1}$ for $n$ large enough. For polynomially growing weights, $\PTt{T_{0n}=n}$ is approximately $n^{-1 + \frac{\gamma}{2(1+\gamma)}}$ for $n$ large enough, which means that the event $\{T_{0n}=n\}$ is even more likely. Thus, at a first glance it seems promising to compare the distributions of $(C_1^n, C_2^n, ..., C_b^n)$ and $(Z_1, Z_2, ..., Z_b)$.

When $\Sn$ is equipped with the uniform or Ewens measure, not only the Conditioning Relation \eqref{eq:tot_var_dist_cond_rel} holds, but additionally the approximating random variables $Z_m$ satisfy the so-called \textit{Logarithmic Condition} 
\begin{align}\label{eq:tot_var_dist_log_con}
m \, \mathbb{E}[Z_m] \rightarrow \vartheta, \quad \quad \text{ as } m \rightarrow \infty.
\end{align}
Many well known combinatorial objects which decompose into elementary components (permutations decompose into cycles, graphs into connected components, polynomials into irreducible factors) satisfy the Conditioning Relation and the Logarithmic Condition (see \cite[Chapter 2]{ABT02} for
a comprehensive overview of examples of logarithmic and non-logarithmic combinatorial structures). 
%
For this class of objects, Arratia et al. \cite{ArBaTa00} developed a unified approach to study the total variation distance \eqref{eq:tot_var_dist_d_bn} only using the Conditioning Relation and the Logarithmic Condition. By the independence of the random variables $Z_m$, Arratia and Tavar\'e \cite{ArTa92} rewrite the right-hand side of \eqref{eq:tot_var_dist_d_bn_one_dim_rv} as 
\begin{align}\label{eq:algrowth_d_TV}
d_b(n)
&= \sum_{k\geq 0} \big(\PTt{T_{0b}=k} - \PTt{T_{0b}=k | T_{0n}=n} \big)^+ \nonumber\\
&= \sum_{k\geq 0} \PTt{T_{0b}=k}\bigg(1 - \frac{\PTt{T_{bn}=n-k}}{\PTt{T_{0n}=n}} \bigg)^+ .
\end{align}
The key to the the analysis of the accuracy of the approximation is some local limit approximation of the distribution of $T_{bn} = \sum_{m=b+1}^n m \, Z_m$.  In \cite{ArBaTa00} it is shown that the Logarithmic Condition ensures that $n^{-1} T_{bn} \rightarrow X_{\vartheta}$ in distribution, where $X_{\vartheta}$ is a random variable only depending on $\vartheta$ and $b = o(n)$. Via this limiting behavior they establish
$$k \mathbb{P}[T_{bn}=k] \sim \vartheta \, \mathbb{P}[k-n \leq T_{bn} \leq k-b],$$
which provides the required local limit approximation. Then their main result (\cite[Theorem 3.1]{ArBaTa00}) is that for all combinatorial structures satisfying \eqref{eq:tot_var_dist_cond_rel} and \eqref{eq:tot_var_dist_log_con}, considered with respect to the Ewens measure, the following holds: 
$$d_b(n) = d_{TV}\big(\mathcal{L}(C_1^n, C_2^n, ..., C_b^n), \mathcal{L}(Z_1, Z_2, ..., Z_b) \big) \rightarrow 0 \quad \text{ for } b = o(n).$$
%
%

%

In this paper we consider random permutations with respect to a weighted measure with parameters $\theta_m = m^{\gamma}$. As mentioned before, the Feller coupling is not available in this situation. Recall Remark~\ref{remark:C_b-for-fixed-b-converge-to-indep-Poisson} and the estimate for $h_n$ given in \eqref{eq:intro_algrowth_h_n}. This implies that the convergence in \eqref{eq:intro_convergence_fixed_b} holds, where the $Z_m$ are independent Poisson random variables with mean 
$$\ETt{Z_m} = \frac{\theta_m}{m} t^m = m^{\gamma-1} t^m,$$
and
\begin{align}\label{eq:algrowth_t_eta}
t = \exp(-\eta_{\gamma}) \quad \text{ with } \quad \eta_{\gamma} = \Big( \frac{n}{\Gamma(1+\gamma)}\Big)^{-\frac{1}{1+\gamma}} .
\end{align}
Unfortunately, the Logarithmic Condition \eqref{eq:tot_var_dist_log_con} is clearly not satisfied, and thus a different approach is needed to prove Theorem~\ref{thm:algrowth_total_variation}. The starting point is equation \eqref{eq:algrowth_d_TV}. We will show that $T_{0b}$, properly rescaled, can be approximated by a Gaussian random variable $G_{0b}$ with appropriately chosen mean and variance. This enables us to prove that the sum $\sum \PTt{T_{0b}=k}$ converges to zero outside a small interval around the mean of $T_{0b}$. Within this interval, we will show that the quotient $\PTt{T_{bn}=n-k}/ \PTt{T_{0n}=n}$ converges to $1$. 
Let us first compute 
$$\mu_{0b}:=\ETt{T_{0b}}, \,\, \mu_{bn}:=\ETt{T_{bn}}, \,\, \sigma^2_{0b}:=\VTt{T_{0b}} \,\, \text{ and } \,\, \sigma^2_{bn}:=\VTt{T_{bn}}.$$ 
\begin{lemma}\label{lem:mu_0b and others}
Recall $\Sigma_2$ as in \eqref{eq:def_sigma}. For $b = o(n^{\frac{1}{1+\gamma}})$ we have 
\begin{enumerate}
\item 
$\mu_{0b} 
 = \frac{1}{1+\gamma} b^{1+\gamma} - \frac{n}{\Gamma(1+\gamma)} \Sigma_2(1+\gamma,b\eta_{\gamma}) + O(b^{\gamma}),$ \\
 \item 
 $\sigma^2_{0b} 
 = \frac{1}{2+\gamma} b^{2+\gamma} - \big(\frac{n}{\Gamma(1+\gamma)}\big)^{\frac{2+\gamma}{1+\gamma}} \Sigma_2(2+\gamma,b\eta_{\gamma}) + O(b^{1+\gamma}),$\\
 \item
 $\mu_{bn} 
= n - \frac{1}{1+\gamma} b^{1+\gamma} + \frac{n}{\Gamma(1+\gamma)} \Sigma_2(1+\gamma,b\eta_{\gamma}) + O(n^{\frac{\gamma}{1+\gamma}}),$\\
\item
$\sigma^2_{bn} 
 = \frac{1+\gamma}{\Gamma(1+\gamma)^{\frac{1}{1+\gamma}}} n^{\frac{2+\gamma}{1+\gamma}} - \frac{1}{2+\gamma} b^{2+\gamma} + \big(\frac{n}{\Gamma(1+\gamma)}\big)^{\frac{2+\gamma}{1+\gamma}} \Sigma_2(2+\gamma,b\eta_{\gamma}) + O(n).$
\end{enumerate}
\end{lemma}
\begin{proof}
Recall \eqref{eq:intro_asym_Gamma_0}, \eqref{eq:intro_asym_Gamma_infty} and \eqref{eq:intro_euler_maclaurin}.
Then $\mu_{0b} = \ETt{T_{0b}}$ is given by
\begin{align*}
\mu_{0b} = \sum_{k=1}^b \theta_k t^k = \sum_{k=1}^b k^{\gamma} t^k
\end{align*}
and  (\ref{eq:intro_euler_maclaurin}) yields
\begin{align*}
\sum_{k=1}^b k^{\gamma}t^k
&= \int_1^b x^{\gamma}t^x dx + \gamma\int_1^b (x-\lfloor x \rfloor) x^{\gamma-1}t^x dx  \\
&+ \log(t)\int_1^b (x-\lfloor x \rfloor) x^{\gamma}t^x dx+ b^{\gamma}t^b (b-\lfloor b \rfloor)  .
\end{align*}
For the first integral, set $t = \exp(-\eta_{\gamma})$ with $\eta_{\gamma}$ as in (\ref{eq:algrowth_t_eta}). With a variable substitution $y = x \eta_{\gamma}$, we obtain
\begin{align*}
\int_1^b x^{\gamma} e^{-x\eta_{\gamma} } dx
&= \frac{n}{\Gamma(1+\gamma)} \int_{\eta_{\gamma}}^{b\eta_{\gamma}} y^{\gamma} e^{-y} dy \nonumber\\
&= \frac{n}{\Gamma(1+\gamma)} \big(\Gamma(1+\gamma, \eta_{\gamma}) - \Gamma(1+\gamma, b\eta_{\gamma}) \big) \nonumber\\
& = \frac{1}{1+\gamma} b^{1+\gamma} - \frac{n}{\Gamma(1+\gamma)} \Sigma_2(1+\gamma,b\eta_{\gamma}) + O(1),
\end{align*}
where the last step follows from (\ref{eq:intro_asym_Gamma_0}) and $b = o(n^{\frac{1}{1+\gamma}})$. 
For the remaining terms one can show that they are of order $O(b^{\gamma})$, which yields assertion $(1)$.
Similarly, 
\begin{align*}
\sigma^2_{0b} = \sum_{k=1}^b k \theta_k t^k = \sum_{k=1}^b k^{1+\gamma} t^k
\end{align*}
and by \eqref{eq:intro_euler_maclaurin}
\begin{align*}
\sigma^2_{0b} 
&= \int_1^b x^{\gamma+1} e^{-x \eta_{\gamma}} dx + O(b^{1+\gamma}) \nonumber\\
&= \Big(\frac{n}{\Gamma(1+\gamma)}\Big)^{\frac{2+\gamma}{1+\gamma}} \big(\Gamma(2+\gamma, \eta_{\gamma}) - \Gamma(2+\gamma, b\eta_{\gamma}) \big) + O(b^{1+\gamma}) \nonumber\\
& = \frac{1}{2+\gamma} b^{2+\gamma} - \Big(\frac{n}{\Gamma(1+\gamma)}\Big)^{\frac{2+\gamma}{1+\gamma}} \Sigma_2(2+\gamma,b\eta_{\gamma}) + O(b^{1+\gamma}),
\end{align*}
proving $(2)$. The computations for $T_{bn}$ are analogues. In particular, notice that
\begin{align}\label{eq:tot_var_distance_mean_bn=n-mean_0b}
\mu_{0b} + \mu_{bn} = \mu_{0n} =  n + O(1).
\end{align}
The proof is complete.
\end{proof}
%
%
%
 %
%
%
%
%
%
%
For two real random variables $X$ and $Y$ with distributions $\mu$ and $\nu$, recall that the Kolmogorov distance $d_K(X,Y)$ is defined by
$$d_K(X,Y) := d_K(\mu,\nu) 
:= \sup_{x \in \R} |\mathbb{P}(X \leq x) - \mathbb{P}(Y \leq x)|.$$
Now define for $x=1+ \frac{\gamma}{3}$
\begin{align}\label{eq:T_bn^x}
T_{0b}^x := \frac{T_{0b}}{ b^{x}},
\quad 
\mu_{0b^x}  
 :=  \frac{\mu_{0b}}{b^x}
\quad \text{ and } \quad
 \sigma_{0b^x} 
 := \frac{\sigma_{0b}}{b^{x}}.
\end{align}
\begin{lemma}\label{lem:algrowth_total_variation_lemma_T_0b_Kol_distance_G_0b}
 Assume $b = o(n^{\frac{1}{1+\gamma}})$ and let $G_{0b}$ be a Gaussian random variable with mean $\mu_{0b^x}$ and variance $\sigma_{0b^x}$. Then
\begin{align}\label{eq:total_var_distance_d_K}
d_K(T^x_{0b}, G_{0b}) =  O\big(\sigma_{0b^x}^{-1}\big) = O\big(b^{-\gamma/6}\big).
\end{align}
\end{lemma}
\begin{proof}
We will show that $T_{0b}^x$ is mod-Gaussian convergent with parameters $\mu_{0b^x}$ and $\sigma_{0b^x}$ (see \cite[Definition 1.1]{JaKoNi09} for the definition of mod-Gaussian convergence).
Then the assertion of the lemma is a direct consequence of  \cite[Remark 3]{KoNi09a}.

The characteristic function of $T_{0b}$ is given by
\begin{align*}
 \E_{\Theta}^t[e^{isT_{0b}}] 
&= \exp\Big(\sum_{k=1}^b\frac{\theta_k}{k}t^k(e^{isk}-1)\Big)\\
&= \exp\Big( is\sum_{k=1}^b \theta_k t^k - \frac{s^2}{2} \sum_{k=1}^b k \theta_k t^k - \frac{i s^3}{6} \sum_{k=1}^b k^2 \theta_k t^k + O\big(s^4\sum_{k=1}^b k^3\theta_k t^k \big)\Big)
\end{align*}
and we need to find an appropriate scaling such that the third term converges to a constant and the error term converges to zero. In Lemma~\ref{lem:mu_0b and others} we have given 
\begin{align*}
\mu_{0b} = \sum_{k=1}^b \theta_k t^k 
\quad \text{ and } \quad
\sigma^2_{0b} = \sum_{k=1}^b k\theta_k t^k .
\end{align*}
Similarly, we compute 
\begin{align*}
\sum_{k=1}^b k^2 \theta_k t^k 
&= \int_1^b x^{\gamma+2} e^{-x \eta_{\gamma}} dx + O(b^{2+\gamma})\\
&= \Big( \frac{n}{\Gamma(1+\gamma)}\Big)^{\frac{3+\gamma}{1+\gamma}} \big(\Gamma(3+\gamma, \eta_{\gamma}) - \Gamma(3+\gamma, b\eta_{\gamma}) \big)\\
& = \frac{1}{3+\gamma} b^{3+\gamma} - \Big(\frac{n}{\Gamma(1+\gamma)}\Big)^{\frac{3+\gamma}{1+\gamma}} \Sigma_2(3+\gamma,b\eta_{\gamma}) + O(1)
\end{align*}
and
\begin{align*}
\sum_{k=1}^b k^3\theta_k t^k
= O(b^{4+\gamma}).
\end{align*}
We therefore have to rescale by $s^x = s/b^x$ such that $b^{3+\gamma -3x}$ converges to a constant. Thus, choose $x = 1+ \frac{\gamma}{3}$, then for $T_{0b}^x := T_{0b}/b^x$ we get
\begin{align*}
 \E_{\Theta}^t[e^{isT_{0b}^x}] 
= \exp\Big( is \mu_{0b^x} - \frac{s^2}{2} \sigma_{0b^x}^2 - \frac{is^3}{6} \delta_{0b^x}+ O\big(s^4 b^{-\frac{\gamma}{3}}\big)\Big)
\end{align*}
where $ \mu_{0b^x}  = \mu_{0b} / b^x $, $\sigma_{0b^x}^2 = \sigma^2_{0b} / b^{2x}$
and 
 \begin{align*}
 \delta_{0b^x}
& = b^{-3(1+\frac{\gamma}{3})}  \sum_{k=1}^b k^{\gamma+2} t^k 
=  \frac{1}{3+\gamma} + O(b n^{-\frac{1}{1+\gamma}}).
\end{align*}
This completes the proof.
\end{proof}
With these preliminary results at hand, we are prepared to prove Theorem~\ref{thm:algrowth_total_variation}.
\begin{proof}[Proof of Theorem~\ref{thm:algrowth_total_variation}]
Assume first $b = o(n^{\frac{1}{1+\gamma}})$ and recall equation \eqref{eq:algrowth_d_TV}. Since the $(..)^+$-term in \eqref{eq:algrowth_d_TV} satisfies $(..)^+ \leq 1$, we want to find $\alpha, \beta $ such that both sums
\begin{align*}
\sum_{k=0}^{\alpha} \mathbb{P}_{\Theta}^t [T_{0b}=k] 
\quad \text{ and } \quad
\sum_{k=\beta}^{\infty} \mathbb{P}_{\Theta}^t [T_{0b}=k]
\end{align*}
converge to zero. Recall the definition of $T_{0b}^x$, $\mu_{0b^x}$ and $\sigma_{0b^x}$ in \eqref{eq:T_bn^x}. As in Lemma~\ref{lem:algrowth_total_variation_lemma_T_0b_Kol_distance_G_0b}, denote by $G_{0b}$ a Gaussian random variable with mean $\mu_{0b^x}$ and standard deviation $\sigma_{0b^x}$. Let $g$ be any function with $g(b) \rightarrow \infty$ as $b \rightarrow \infty$ and define 
$$\epsilon_b := \sigma_{0b} \, g(b).$$
Then, as $n \rightarrow \infty$, we have
\begin{align*}
\mathbb{P}_{\Theta}^t \big[\mu_{0b} - \epsilon_b \leq T_{0b}\leq \mu_{0b} + \epsilon_b\big] \rightarrow 1. 
\end{align*}
To see this, notice that for $\epsilon_b^x := \epsilon_b / b^x$,
\begin{align*}
&\mathbb{P}_{\Theta}^t \big[\mu_{0b} - \epsilon_b \leq T_{0b}\leq \mu_{0b} + \epsilon_b\big]  \\
=&\, \mathbb{P}_{\Theta}^t \big[\mu_{0b^x} - \epsilon^x_b \leq T^x_{0b}\leq \mu_{0b^x} + \epsilon^x_b\big] \\
=&\, \mathbb{P}_{\Theta}^t \big[\mu_{0b^x} - \epsilon^x_b \leq G_{0b}\leq \mu_{0b^x} + \epsilon^x_b\big] + O\big(d_K(T_{0b}^x,G_b)\big) .
\end{align*}
Now Lemma~\ref{lem:algrowth_total_variation_lemma_T_0b_Kol_distance_G_0b} yields $d_K(T_{0b}^x,G_b) = O\big(b^{-\gamma/6}\big).$ By basic properties of the Gaussian distribution, 
\begin{align*}
\mathbb{P}_{\Theta}^t \big[\mu_{0b^x} - \epsilon^x_b\leq G_{0b}\leq \mu_{0b^x} + \epsilon^x_b\big] 
&= \frac{1}{2} \Big(\erf\Big(\frac{\epsilon^x_b}{\sqrt{2}\sigma_{0b^x}}\Big) - \erf\Big(-\frac{\epsilon^x_b}{\sqrt{2}\sigma_{0b^x}}\Big)  \Big) \\
&= \frac{1}{2} \Big(\erf\Big(\frac{g(b)}{\sqrt{2}}\Big) - \erf\Big(-\frac{g(b)}{\sqrt{2}}\Big)  \Big) ,
\end{align*}
where $\erf (x)$ denotes the error function, which satisfies the asymptotics \eqref{eq:intro_erf_infty} and \eqref{eq:intro_erf_minus_infty}. 
Thus, as $n \rightarrow \infty$, for all $g$ with $g(b) \rightarrow \infty$, 
\begin{align}\label{eq:tot_var_distance_T_0b_gegen_null}
\mathbb{P}_{\Theta}^t \big[\mu_{0b} - \epsilon_b \leq T_{0b}\leq \mu_{0b} + \epsilon_b\big]
= 1 + O\big( g^{-1}(b) e^{-g^2(b)} + b^{-\gamma/6}\big)
\end{align}
and therefore both sums
\begin{align*}
\sum_{k=0}^{\mu_{0b} - \epsilon_b} \mathbb{P}_{\Theta}^t [T_{0b}=k] 
\quad \text{ and } \quad
\sum_{k= \mu_{0b} + \epsilon_b}^{\infty} \mathbb{P}_{\Theta}^t [T_{0b}=k] 
\end{align*}
are of order $O\big( g^{-1}(b) e^{-g^2(b)} + b^{-\gamma/6}\big)$.
Next, in view of \eqref{eq:algrowth_d_TV}, we have to show that the sum
\begin{align}\label{eq:total_var_distance_last_sum}
\sum_{k= \mu_{0b} - g(b) \, \sigma_{0b}}^{ \mu_{0b} + g(b) \, \sigma_{0b}} \mathbb{P}_{\Theta}^t [T_{0b}=k] \bigg(1 - \frac{\PTt{T_{bn}=n-k}}{\PTt{T_{0n}=n}} \bigg)^+ 
\end{align}
converges to zero. Recall \eqref{eq:tot_var_distance_mean_bn=n-mean_0b} : $\mu_{bn} = n - \mu_{0b}$, and denote $I_b := [-g(b) \, \sigma_{0b}, g(b) \, \sigma_{0b}]$. Then we can rewrite the previous sum as 
\begin{align}\label{eq:tot_var_sup}
\eqref{eq:total_var_distance_last_sum} 
&= \sum_{j \in I_b} \mathbb{P}_{\Theta}^t [T_{0b}=\mu_{0b} -j] \bigg(1 - \frac{\PTt{T_{bn}=\mu_{bn}+j}}{\PTt{T_{0n}=n}} \bigg)^+ \nonumber\\
&\leq \sup_{j \in I_b} \bigg(1 - \frac{\PTt{T_{bn}=\mu_{bn} + j}}{\PTt{T_{0n}=n}} \bigg)^+ 
\end{align}
and we have to show that this term converges to $0$. 

Let us first give an heuristic argument why this should be true. First, one can show that
$$\frac{\PTt{T_{bn}=\mu_{bn} + j}}{\PTt{T_{0n}=n}} \rightarrow 1
\quad \text{ if and only if } \quad
\frac{\PTt{T_{bn}=\mu_{bn} + j}}{\PTt{T_{bn}=\mu_{bn}}} \rightarrow 1.$$ Similarly to \eqref{eq:T_bn^x} and Lemma~\ref{lem:algrowth_total_variation_lemma_T_0b_Kol_distance_G_0b}, we can show that $T_{bn}^y := T_{bn} / n^y$ with $y = \frac{3+\gamma}{3(1+\gamma)}$ is approximately Gaussian with mean $\mu^y_{bn}:= \mu_{bn}/n^y$ and standard deviation $\sigma^y_{bn}:= \sigma_{bn}/n^y$. Thus, vaguely, let us consider for a moment that $T_{bn}$ is approximately (a discrete version of a) Gaussian random variable $G_{bn}$ with mean $\mu_{bn}$ and variance $\sigma^2_{bn}$. Then, for $\delta = o(j)$, the question is for which $j$ the following holds: 
$$\PTt{\mu_{bn} + j \leq G_{bn} \leq \mu_{bn} + j + \delta} 
\sim \PTt{\mu_{bn} \leq G_{bn} \leq \mu_{bn} + \delta}.$$
By the standard properties of the Gaussian distribution, this holds for any $j = o(\sigma_{bn})$. Thus, the crucial point why \eqref{eq:tot_var_sup} should converge to zero is that
\begin{align}\label{eq:tot_var_sigma_0b_klein_o_sigma_bn}
|j| \leq g(b) \, \sigma_{0b} = o(\sigma_{bn}).
\end{align}
We have $\sigma_{0b} = o(\sigma_{bn})$ and since $g(b) \rightarrow \infty$ may be chosen arbitrarily this implies $g(b) \, \sigma_{0b} = o(\sigma_{bn})$:
$$\frac{g(b) \, \sigma_{0b}}{\sigma_{bn}} = O \big(b^{\frac{2+\gamma}{2}} n^{- \frac{2+\gamma}{2(1+\gamma)}} g(b) \big)$$
and now choose 
\begin{align}\label{eq:tot_var_distance_gb}
g(b) := \big( n^{\frac{1}{1+\gamma}} b^{-1} \big)^{\frac{\gamma}{2}}
\end{align}
to get
$$\frac{g(b) \, \sigma_{0b}}{\sigma_{bn}} = O \big(b n^{- \frac{1}{1+\gamma}} \big) \rightarrow 0.$$

For the rigorous proof that \eqref{eq:tot_var_sup} converges to $0$, we compute $\PTt{T_{bn}=\mu_{bn} + j}$ explicitly by means of saddle-point analysis. We have,
\begin{align*}
 \E_{\Theta}^t[u^{T_{bn}}] 
= \prod_{k=b+1}^n \ETt{u^{k Z_k}} 
= \exp\Big(\sum_{k=b+1}^n \frac{\theta_k}{k}t^k(u^k-1)\Big),
\end{align*}
where $t$ is as in \eqref{eq:algrowth_t_eta}. Now, for $m \leq n$,
\begin{align*}
 \mathbb{P}_{\Theta}^t[T_{bn}=m] 
 &=  e^{- S_{b}(t)} t^m [u^m] \exp\Big(\sum_{k=b+1}^n k^{\gamma -1} u^k\Big)  \\
  &=  e^{- S_{b}(t)} t^m [u^m] \exp\Big(\sum_{k=b+1}^{\infty} k^{\gamma -1} u^k\Big),
\end{align*}
where $S_{b}(t):=\sum_{k=b+1}^n k^{\gamma-1} t^k$. Notice that $\mu_{bn} + j \leq n$ for $n$ large with the above chosen $g(n)$. In particular, for $b =0$ and $m=n$, we get
\begin{align*}
 \mathbb{P}_{\Theta}^t[T_{0n}=n] 
 =  e^{- S_{0}(t)} t^n [u^n] \exp\big(g_{\Theta}(u)\big) 
  =  e^{- S_0(t)} t^n h_n,
\end{align*}
where $h_n$ is as in \eqref{eq:intro_algrowth_h_n}. For $b \neq 0$, to prove that $g_{\Theta,b}(u) := \sum_{k=b+1}^{\infty} k^{\gamma -1} u^k$ is log-admissible one proceeds along the same lines as in the proof of Lemma~\ref{lem:algrowth_g(r,s)_is_log_adm}. The leading term of the saddle point solution
$$\alpha(r_m) = m + o(\sqrt{\beta(r_m)})$$
is given by $r_m = \exp(-v_m)$ with 
\begin{align*}
v_m =  \Big(\frac{m + \mu_{0b}}{\Gamma(1+\gamma)}\Big)^{- \frac{1}{1+\gamma}}.
\end{align*}
Lemma~\ref{lem:algrowth_expansion Gns} together with Remark~\ref{rem:complete_asymptotic} yields
\begin{align*}
[u^{m}] \exp\Big(\sum_{k=b+1}^{\infty} k^{\gamma -1} u^k\Big)
&= \frac{1}{\sqrt{2 \pi \beta(r_{m})}}  \exp\big(g_{\Theta,b}(r_{m}) + m \, v_m\big) \big(1+ O(n^{-\frac{\gamma}{1+\gamma}} )\big).
\end{align*}
Thus we have
\begin{align*}
& \frac{ \mathbb{P}_{\Theta}^t[T_{bn}=m] }{ \mathbb{P}_{\Theta}^t[T_{0n}=n] } \\
= &\, \frac{t^{m-n}}{h_n \, \sqrt{2 \pi \beta(r_m)}}  \exp\Big(g_{\Theta,b}(r_{m}) + m\, v_m + \sum_{k=1}^b k^{\gamma-1}t^k\Big)  \big(1+ O(n^{-\frac{\gamma}{1+\gamma}} )\big).
\end{align*}
Recall $t= \exp(- \eta_{\gamma})$ with $\eta_{\gamma} = \big(\frac{n}{\Gamma(1+\gamma)}\big)^{-\frac{1}{1+\gamma}}$. Let us first compute $\beta(r_{m})$. Similarly as in the proof of Lemma~\ref{lem:mu_0b and others}, we have
\begin{align*}
\beta(r_{m}) 
&= \sum_{k=b+1}^{\infty} k^{\gamma+1} r_{m}^{k}
= \sum_{k=1}^{\infty} k^{\gamma+1} r_{m}^{k} - \sum_{k=1}^{b} k^{\gamma+1} r_{m}^{k} \\
&= \Li_{-\gamma-1}(r_m) -  \int_{1}^{b} x^{\gamma+1} e^{-x v_{\widetilde{m}}} dx + O(b^{1+\gamma}) \\
&= \Gamma(2+\gamma) v_m^{-(2+\gamma)} + O\big( b^{2+\gamma} \big).
\end{align*}
Together with $h_n$ as in \eqref{eq:intro_algrowth_h_n} we get
\begin{align*}
\frac{ \mathbb{P}_{\Theta}^t[T_{bn}=m] }{ \mathbb{P}_{\Theta}^t[T_{0n}=n] }
= H_{\Theta,b}(r_{m},t)\exp\big(G_{\Theta,b}(r_{m},t) \big) \big(1+ O(n^{-\frac{\gamma}{1+\gamma}} ) \big),
\end{align*}
with 
\begin{align*}
H_{\Theta,b}(r_{m},t)
 = \big(v_m^{-(2+\gamma)} + O(b^{2+\gamma})\big)^{-1/2} \Big(\frac{n}{\Gamma(1+\gamma)} \Big)^{\frac{2+\gamma}{2(1+\gamma)}}
\end{align*}
and 
\begin{align*}
G_{\Theta,b}(r_{m},t)
 = g_{\Theta,b}(r_{m}) + m\, (v_m - \eta_{\gamma}) + \sum_{k=1}^b k^{\gamma-1}t^k - \frac{n\, \eta_{\gamma}}{\gamma} - \zeta(1-\gamma).
\end{align*}
Recall that we are interested in $\widetilde{m} := \mu_{bn} +j$ with $j \in I_b$ so that
\begin{align*}
v_{\widetilde{m}} 
=  \Big(\frac{n + j}{\Gamma(1+\gamma)}\Big)^{- \frac{1}{1+\gamma}} 
= \eta_{\gamma} -  \frac{ \Gamma(1+\gamma)^{\frac{1}{1+\gamma}} }{1+\gamma} j n^{-\frac{2+\gamma}{1+\gamma}} + O\big(j^2 n^{-\frac{3+2\gamma}{1+\gamma}}\big).
\end{align*}
Thus,
\begin{align*}
H_{\Theta,b}(r_{\widetilde{m}},t)
& = \bigg(\Big(\frac{n +j}{\Gamma(1+\gamma)} \Big)^{-\frac{2+\gamma}{2(1+\gamma)}} + O\big(n^{-\frac{3(2+\gamma)}{2(1+\gamma)}}b^{2+\gamma}\big)\bigg) \Big(\frac{n}{\Gamma(1+\gamma)} \Big)^{\frac{2+\gamma}{2(1+\gamma)}} \\
& = 1+  O\Big( j n^{-1} + n^{-\frac{2(2+\gamma)}{2(1+\gamma)}}b^{2+\gamma}\Big),
\end{align*}
and the error term converges to zero since for $g(b)$ as in \eqref{eq:tot_var_distance_gb} we have  $|j| \leq g(b) \, \sigma_{0b} = o(n)$.
It remains to compute $G_{\Theta,b}(r_{\widetilde{m}},t)$. First notice that
\begin{align*}
\widetilde{m}(v_{\widetilde{m}} - \eta_{\gamma})
= - \mu_{bn}  \frac{ \Gamma(1+\gamma)^{\frac{1}{1+\gamma}} }{1+\gamma} j n^{-\frac{2+\gamma}{1+\gamma}}
+ O\big(j^2 n^{-\frac{2+\gamma}{1+\gamma}} \big),
\end{align*}
 and $\mu_{bn} = n -\mu_{0b}$. Furthermore, 
\begin{align*}
g_{\Theta,b}(r_{m})
&= \sum_{k=b+1}^{\infty} k^{\gamma-1} r_{m}^{k}
= \sum_{k=1}^{\infty} k^{\gamma-1} r_{m}^{k} - \sum_{k=1}^{b} k^{\gamma-1} r_{m}^{k} \\
&=\Li_{1-\gamma}(r_m) - \sum_{k=1}^{b} k^{\gamma-1} r_{m}^{k},
\end{align*}
where $\Li$ denotes the polylogarithm as in \eqref{eq:algrowth_asypm_polylog} and 
\begin{align*}
t^k - r_{\widetilde{m}}^{k} 
= e^{-k \, \eta_{\gamma}} 
\bigg( 1-\exp\Big(  \frac{ \Gamma(1+\gamma)^{\frac{1}{1+\gamma}} }{1+\gamma} kj n^{-\frac{2+\gamma}{1+\gamma}} + O(j^2 n^{-\frac{3+2\gamma}{1+\gamma}}) \Big) \bigg).
\end{align*}
Then for $k\leq b$ we have $k j n^{-1-\frac{1}{1+\gamma}} = o(1)$ and this yields
\begin{align*}
t^k - r_{\widetilde{m}}^{k} 
= e^{-k \, \eta_{\gamma}} 
\bigg( -  \frac{ \Gamma(1+\gamma)^{\frac{1}{1+\gamma}} }{1+\gamma} kj n^{-\frac{2+\gamma}{1+\gamma}}
+ O(k j^2 n^{-\frac{3+2\gamma}{1+\gamma}}) \bigg).
\end{align*}
Thus,
\begin{align*}
\sum_{k=1}^{b} k^{\gamma-1}(t^k - r_{\widetilde{m}}^{k}) 
&= -  \frac{ \Gamma(1+\gamma)^{\frac{1}{1+\gamma}} }{1+\gamma} j n^{-\frac{2+\gamma}{1+\gamma}} \sum_{k=1}^{b} k^{\gamma}t^k 
+ O \bigg(j^2 n^{-\frac{3+2\gamma}{1+\gamma}}\sum_{k=1}^{b} k^{\gamma}t^k \bigg)\\
 &= -  \frac{ \Gamma(1+\gamma)^{\frac{1}{1+\gamma}} }{1+\gamma} j n^{-\frac{2+\gamma}{1+\gamma}} \mu_{0b} 
+ O \bigg(j^2 n^{-\frac{3+2\gamma}{1+\gamma}}\mu_{0b} \bigg),
\end{align*}
and notice that the error term converges to zero. Altogether, we have proved so far 
\begin{align*}
G_{\Theta,b}(r_{\widetilde{m}},t)
& = g_{\Theta,b}(r_{\widetilde{m}}) + \widetilde{m}\, (v_{\widetilde{m}} - \eta_{\gamma}) + \sum_{k=1}^b k^{\gamma-1}t^k - \frac{n\, \eta_{\gamma}}{\gamma} - \zeta(1-\gamma)\\
& = \Li_{1-\gamma}(r_{\widetilde{m}})  + \widetilde{m}\, (v_{\widetilde{m}} - \eta_{\gamma}) + \sum_{k=1}^b k^{\gamma-1}(t^k-r_{\widetilde{m}}^k) - \frac{n\, \eta_{\gamma}}{\gamma} - \zeta(1-\gamma)\\
& = \Li_{1-\gamma}(r_{\widetilde{m}}) 
- \frac{ \Gamma(1+\gamma)^{\frac{1}{1+\gamma}} }{1+\gamma} j n^{-\frac{1}{1+\gamma}} 
- \frac{n\, \eta_{\gamma}}{\gamma} -\zeta(1-\gamma) + O\big(j^2 n^{-\frac{2+\gamma}{1+\gamma}} \big).
\end{align*}
Finally, we have
\begin{align*}
\Li_{1-\gamma}(r_m) 
&= \Gamma(\gamma) v_m^{-\gamma} + \zeta(1-\gamma) + O(v_m) \\
&= \frac{\Gamma(\gamma)}{\Gamma(1+\gamma)^{\frac{\gamma}{1+\gamma}}} \Big( n^{\frac{\gamma}{1+\gamma}} + \frac{\gamma}{1+\gamma} j n^{-\frac{1}{1+\gamma}} + O( j^2  n^{-\frac{2+\gamma}{1+\gamma}}) \Big) + \zeta(1-\gamma)
\end{align*}
which yields
\begin{align}\label{eq:tot_var_distance_error_term}
G_{\Theta,b}(r_{\widetilde{m}},t)
 =  O\big(j^2 n^{-\frac{2+\gamma}{1+\gamma}} \big) 
 = O\bigg(\frac{g^2(b) \, \sigma^2_{0b}}{\sigma^2_{bn}}\bigg),
\end{align}
and this converges to zero because of \eqref{eq:tot_var_sigma_0b_klein_o_sigma_bn}.
Altogether, we have proved that if $b = o(n^{\frac{1}{1+\gamma}})$ then 
$$d_b(n) = O\big( b^{2+\gamma} n^{-\frac{2+\gamma}{1+\gamma}}   + b^{-\gamma/6} + O(n^{-\frac{\gamma}{1+\gamma}} ) \big).$$
To complete the proof of Theorem~\ref{thm:algrowth_total_variation}, we assume now $b \neq o(n^{\frac{1}{1+\gamma}})$ and show that in this case $\liminf_{n \rightarrow \infty} d_b(n)>0$. 
Recall from (\ref{eq:algrowth_d_TV}) that
\begin{align*}
d_b(n) \geq \mathbb{P}_{\Theta}^t [T_{0b}>n].
\end{align*}

For $b \, n^{-\frac{1}{1+\gamma}} \rightarrow \infty$ 
the mean of $T_{0b}$ is $n + O(1)$ and the variance is of order $n^{\frac{2+\gamma}{1+\gamma}}$. 
Thus $\mathbb{P}_{\Theta}^t [T_{0b}>n] > 0$ for all $n$. But if $b = c n^{\frac{1}{1+\gamma}}$, then $\ETt{T_{0b}} = Cn + O(1)$ where $C=C(c)$ can be very small when $c$ is very small. 
In particular, if $C<1$, then $\mathbb{P}_{\Theta}^t [T_{0b}>n] \rightarrow 0$, thus a more elaborate argument is needed.

A crucial point in the proof above is equation \eqref{eq:tot_var_sigma_0b_klein_o_sigma_bn}. Notice that for $b = c n^{\frac{1}{1+\gamma}}$ the usual computations give
\begin{align*}
\mu_{0b}' := \ETt{T_{0b}} = O(n)
\quad \text{ and } \quad
\sigma_{0b}' := \sqrt{\VTt{T_{0b}}} = O\big( n^{\frac{2+\gamma}{2(1+\gamma)}}\big)
\end{align*}
as well as
\begin{align*}
\mu_{bn}' := \ETt{T_{bn}} = O(n)
\quad \text{ and } \quad
\sigma_{bn}' := \sqrt{\VTt{T_{bn}}} = O\big(n^{\frac{2+\gamma}{2(1+\gamma)}}\big).
\end{align*}
 Thus, unlike as in \eqref{eq:tot_var_sigma_0b_klein_o_sigma_bn}, here  $\sigma'_{0b} = o(\sigma'_{bn})$ does not hold, but we have $\sigma'_{0b} = O(\sigma'_{bn})$. Therefore, $\mathbb{P}_{\Theta}^t [T_{bn}= \mu'_{bn} - k]/\mathbb{P}_{\Theta}^t [T_{0n}=n]$ will not converge to $1$ implying that $d_b(n)$ will not converge to $0$. In a different setting, this was also proven in \cite{CiZe13}: suppose that $d_b(n)\to 0$ for $b= c n^{\frac{1}{1+\gamma}}$ with $c$ some non-negative constant. Then the random variables
  \begin{align*}
   \sum_{m=\frac{c}{2} n^{\frac{1}{1+\gamma}}}^{c n^{\frac{1}{1+\gamma}}} C_m
   \quad \text{ and } \quad
   \sum_{m=\frac{c}{2} n^{\frac{1}{1+\gamma}}}^{c n^{\frac{1}{1+\gamma}}} Z_m
  \end{align*}
would have same limit as $n\to\infty$. 
However, it was shown in \cite[Theorem~3.6, Theorem~4.6 and Remark~4.4]{CiZe13} that for all $c>0$ these two random variables satisfy two different central limit theorems.
\end{proof}

\begin{remark}
Notice that the term $b^{-\gamma /6}$ in the order of $d_b(n)$ in Theorem~\ref{thm:algrowth_total_variation} comes from the Kolmogorov distance $d_K(T_{0b}^x,G_b)$. Instead of using the Gaussian approximation, one could prove the first part of the theorem also using saddle point analysis to compute $\PTt{T_{0b}=m}$ explicitly. This would give the same result but without the $b^{-\gamma /6}$ term. However, we decided to state the proof using the Gaussian approximation since it allows an intuitive understanding of what is going on.
\end{remark}


\vskip 40pt

\section{The Erd\H{o}s-Tur\'an Law } \label{section:erdos-turan-law}

Recall that the order $O_n(\sigma)$ of a permutation $\sigma \in \Sn$ is the smallest integer $k$ such that the $k$-fold application of $\sigma$ to itself gives the identity. Assume that $\sigma = \sigma_1\cdots \sigma_\ell$ with $\sigma_i$ disjoint cycles of length $\la_i$, then $O_n(\sigma)$ can be computed as
$$O_n(\sigma) = \lcm(\lambda_1, \la_2, \cdots \la_{\ell}).$$
In Section~\ref{subsection:algrowth_preliminaries} an approximating random variable $Y_n$ is introduced which shares many properties with $O_n$ but is much easier to handle. Then Section~\ref{subsection:algrowth_erdoes_turan} is devoted to the proof of Theorem~\ref{thm:algrowth_CLT_logO_n} and Section~\ref{subsection:FCLT} to the proof of Theorem~\ref{thm:fclt}.

\subsection{Preliminaries}\label{subsection:algrowth_preliminaries}
A common approach to investigate the asymptotic behavior of $\log (O_n)$ is to introduce the random variable
\begin{align*}
 Y_n := \prod_{m=1}^{n}m^{C_m} 
\end{align*}
where the $C_m$ denote the cycle counts, and to show that $\log (O_n)$ and $\log (Y_n)$ are relatively close in a certain sense. 
To give explicit expressions for $O_n$ and $Y_n$ involving the cycle counts $C_m$, introduce
\begin{align*}
D_{nk} := \sum_{m=1}^n C_m \one_{\{k | m\}}
\quad \text{ and } \quad
D_{nk}^* := \min \{1, D_{nk} \}.
\end{align*}
Now let $p_1,p_2,\dots$ be the prime numbers
and $q_{m,i}$ be the multiplicity of a prime number $p_i$ in the number $m$. Then
\begin{align}
\label{eq:Yn_with_Cm}
 Y_n 
&= \prod_{m=1}^{n}m^{C_m} 
= \prod_{m=1}^{n}(p_1^{q_{m,1}} p_2^{q_{m,2}} \cdots p_n^{q_{m,n}})^{C_m}\nonumber\\
&= \prod_{i=1}^{n}p_i^{\,C_1\cdot q_{1,i} + C_2\cdot q_{2,i} + \cdots+ C_n \cdot q_{n,i}}
= 
  \prod_{p \leq n}p^{\,\sum_{j=1}^n{D_{np^j}}} ,
\end{align}
where $\prod_{p \leq n}$ denotes the product over all prime numbers that are less or equal to $n$. The last equality can be understood as follows:
let $p$ be fixed and define $m=p^{\,q_{m,i}} \cdot a$ where $a$ and $p$ are coprime (meaning that their least common divisor is $1$). Then $C_m$ occurs exactly once in the sum $D_{np^j}$ if $j\leq q_{m,i}$ and does not occur in $D_{np^j}$ if $j> q_{m,i}$. Thus $C_m$ occurs $q_{m,i}$ times in the sum $\sum_{j=1}^n{D_{np^j}}$.
%
Furthermore, we have also used that $D_{nk} =0$ for $k>n$.
Analogously, we have
\begin{align}
\label{eq:classF_On_with_prims}
O_n = \prod_{p \leq n}p^{\,\sum_{j=1}^n{D^*_{np^j}}} .
\end{align}

To simplify the logarithm of the expressions \eqref{eq:Yn_with_Cm} and \eqref{eq:classF_On_with_prims},  we introduce the von Mangoldt function $\Lambda$, which is defined as
\begin{align*}
\Lambda(n) = 
\begin{cases} 
\log(p) &\mbox{if } n = p^k \text{ for some prime } p \text{ and } k\geq 1,  \\ 
0 & \mbox{otherwise.} 
\end{cases} 
\end{align*}
Then we  obtain
\begin{align}
\label{eq:log_On_von_mangold}
\log Y_n= \sum_{k \leq n} \Lambda(k) D_{nk} 
\quad \text{ and } \quad
\log O_n= \sum_{k \leq n} \Lambda(k) D^*_{nk}.
\end{align}
Now define 
\begin{align}\label{eq:Delta_n}
\Delta_n 
:= \log (Y_n) - \log (O_n) 
 = \sum_{k \leq n}  \Lambda(k) \big(D_{nk} - D^*_{nk}\big) .
  \end{align}
Typically, in order to prove properties of $\log(O_n)$, one first establishes them for $\log(Y_n)$ and then one needs to show that $\Delta_n$ is approximately small enough to transfer the result to $\log(O_n)$, see Lemma~\ref{lem:algrowth_closeness} below.

Recall \eqref{eq:generating_series_logYn_allgemein}. For $\theta_m = m^{\gamma}$ one obtains the generating series 
\begin{align}\label{eq:algrowth_g_Theta}
 \sum_{n=0}^{\infty} h_n \E_{\Theta}[\exp(s \log (Y_n))] t^n  
= \exp\left( \sum_{m=1}^{\infty} \frac{1}{m^{1 - s - \gamma}} t^m  \right) =: \exp (\hat{g}_{\Theta}(t,s)).
\end{align}
As we consider $s$ fixed for the moment, we may write $\hat{g}_{\Theta}(t)$ instead of $\hat{g}_{\Theta}(t,s)$.
The function $\hat{g}_{\Theta}(t)$ is known to be the polylogarithm $\Li_{a}(t)$ with parameter 
$$a = 1 - s -\gamma.$$
For $\gamma >0$ and as $t \rightarrow 1$ it satisfies the asymptotic \eqref{eq:algrowth_asypm_polylog_2}. We will show that $\hat{g}_{\Theta}(t,s)$ is $\log$-admissible (see Definition~\ref{def_log-admissible}) in order to apply Lemma~\ref{lem:algrowth_expansion Gns} to compute $\E_{\Theta}[\exp(s \log (Y_n))]$.

\begin{lemma}
\label{lem:algrowth_g(r,s)_is_log_adm}
$\hat{g}_{\Theta}(t,s)$ is $\log$-admissible for $\gamma > 0$, $s > -\alpha$.
\end{lemma}
\begin{proof}
For $k \geq 1$ as $t \rightarrow 1$ the following holds:
\begin{align}\label{eq:algrowth_g^(k)}
 \hat{g}_{\Theta}^{(k)}(t) 
= t^{-k} \Li_{a -k}(t)
=  \Gamma(1+k -a)(-\log (t))^{a -k-1}t^{-k} + O(1). 
\end{align}
The proof that $\hat{g}_{\Theta}(t,s)$ satisfies the properties given in Definition~\ref{def_log-admissible} is analogous to the proof of Proposition 3.7 in \cite{MaNiZe11}; one simply has to verify that all involved expressions are uniform in $s$ for $-\gamma +\epsilon \leq s \leq C$ for some constant $C$. This is straightforward and we thus omit the details.
\end{proof}

Let us now compute the generating function of $\log (Y_n)$ by means of Lemma~\ref{lem:algrowth_expansion Gns}. 
\begin{theorem}\label{thm:algrowth_mom-gen_logYn}
Let $\hat{g}_{\Theta}$ be as in (\ref{eq:algrowth_g_Theta}) with $\gamma > 0$. Then we have
\begin{align*}
&\E_{\Theta}[\exp(s \log (Y_n))] \\
= & \Big(\sqrt{\tilde{\gamma}_{2,s}} \,  n^{\frac{1}{2}(\frac{1}{1+\gamma} - \frac{1}{1+\gamma+s})}\Big) 
\exp\bigg(\tilde{\gamma}_{1,s} \, n^{1 -\frac{1}{1+\gamma+s}} -\tilde{\gamma}_{1,0} \, n^{1 -\frac{1}{1+\gamma}} \bigg) \\
& \times \exp\big(\zeta(1-s-\gamma) - \zeta(1-\gamma)\big)\big(1+o(1)\big)
\end{align*}
with
\begin{align*}
\tilde{\gamma}_{1,s} = \frac{(1+\gamma +s)\Gamma(\gamma +s)}{\Gamma(1+\gamma+s)^{1-\frac{1}{1+\gamma+s}}} , 
\quad 
\tilde{\gamma}_{2,s} =  \frac{(1+\gamma) \Gamma(1+\gamma +s)^{\frac{1}{1+\gamma+s}}}{(1+\gamma +s) \Gamma(1+\gamma)^{\frac{1}{1+\gamma}}},
\end{align*}
where the error bounds are uniform in $s$ for bounded $s$, $s > -\gamma + \epsilon$.
\end{theorem}
\begin{proof}
We first compute $r_{ns}$. This should satisfy
\begin{align*}
\alpha_s(r_{ns}) =n
\end{align*}
but as stated in Remark~\ref{rem:saddle_point} it actually suffices that
\begin{align}\label{eq:algrowth_saddlepoint}
\alpha_s(r_{ns})  - n = o\Big(\sqrt{\beta_s(r_{ns}) }\Big)
\end{align}
holds. We set for $a=1-s-\gamma$
\begin{align*}
r_{ns} 
= \exp \bigg(-\Big(\frac{n}{\Gamma(2-a)} \Big)^{\frac{1}{a -2}}\bigg) 
\end{align*}
and obtain
\begin{align*}
\alpha_s(r_{ns}) = n + O(1)
\quad \text { and } \quad 
\beta_s(r_{ns}) 
= \Gamma(3-a) \Big(\frac{n}{\Gamma(2-a)} \Big)^{1+\frac{1}{2-a}} +O(1), 
\end{align*}
so that (\ref{eq:algrowth_saddlepoint}) holds. Furthermore,
\begin{align*}
\hat{g}_{\Theta}(r_{ns}, s) 
= \Gamma(1-a) \Big(\frac{n}{\Gamma(2-a)} \Big)^{1-\frac{1}{2-a}} + \zeta(1-s-\gamma) +o(1).
\end{align*}
We now have
\begin{align*}
 G_{n,s} = [t]^n \exp(\hat{g}_{\Theta}(t, s)) = h_n \E_{\Theta}[\exp(s \log (Y_n))].
\end{align*}
Therefore,
\begin{align*}
h_n 
= &\, G_{n,0}=\frac{1}{\sqrt{2\pi}} \big(r_{n0} \big)^{-n} \beta_0(r_{n0})^{-1/2} \exp\big(\hat{g}_{\Theta}(r_{n0},0) \big) \big(1+o(1)\big)\\
= & \, \big(2\pi \Gamma(2+\gamma)\big)^{-\frac{1}{2}} \Big(\frac{\Gamma(1+\gamma)}{n} \Big)^{\frac{2+\gamma}{2(1+\gamma)}}  \times \nonumber\\
&\exp\bigg(\frac{1+\gamma}{\gamma} \,\Gamma(1+\gamma)^{\frac{1}{1+\gamma}} \, n^{\frac{\gamma}{1+\gamma}} + \zeta(1-\gamma) \bigg)\big(1+o(1)\big)
\end{align*}
and
\begin{align*}
&\E_{\Theta}[\exp(s \log (Y_n))] \nonumber\\
&= \bigg(\frac{r_{n0}}{r_{ns}}\bigg)^n \bigg(\frac{\beta_0(r_{n0})}{\beta_s(r_{ns})}\bigg)^{1/2} 
\exp\big(\hat{g}_{\Theta}(r_{ns}, s)-\hat{g}_{\Theta}(r_{n0}, 0)\big)  \big(1+o(1)\big).
\end{align*}
This gives the result.
\end{proof}
\begin{remark}\label{remark_bound_total_var_distance_too_small}
Given Theorem~\ref{thm:algrowth_total_variation}, a natural way to investigate further properties of $\log(O_n)$, for example to prove the central limit theorem, would be to work with the functional $\log(P_n):=\sum_{m=1}^n \log(m)Z_m $ instead of with $\log(Y_n)=\sum_{m=1}^n \log(m)C_m$ and to show that the contribution of the large components $C_{b+1}, ..., C_n$ is negligible. However, in the current setting, the large cycle counts actually do contribute to the behavior of $\log(O_n)$. To see this, one may easily compute the moment generating function of $\log(P_n)$ to show that it satisfies the central limit theorem 
\begin{align*}
\frac{\log (P_n) - \tilde{G}(n)}
{\sqrt{ F(n)}} 
\overset{d}{\longrightarrow} \mathcal{N}(0,1)
\end{align*}
where $F(n)$ is as in Theorem~\ref{thm:algrowth_CLT_logO_n} but 
\begin{align}\label{eq:algrowth_exp_logZ_n}
\tilde{G}(n)= \frac{K(\gamma)}{1+\gamma} \, n^{\frac{\gamma}{1+\gamma}} \log(n) + n^{\frac{\gamma}{1+\gamma}} \tilde{H}(n)
\end{align} 
with
$$\tilde{H}(n) = - K(\gamma)  \frac{\log(\Gamma(1+\gamma))}{1+\gamma} . $$
Thus, even rescaled by $F(n)$, the discrepancy between $G(n)$ and $\tilde{G}(n)$ is too large to prove the central limit theorem for $\log(O_n)$ via the independent approximating process. More generally, it seems that the bound $b=o(n^{\frac{1}{1+\gamma}})$ is too small to exploit Theorem~\ref{thm:algrowth_total_variation} to study the whole cycle count process. Nonetheless, in Section~\ref{subsection:FCLT} we will explain how to use Theorem~\ref{thm:algrowth_total_variation} in order to investigate properties of the small cycles. 
\end{remark}


\vskip 15pt 

\subsection{Proof of Theorem~\ref{thm:algrowth_CLT_logO_n}}\label{subsection:algrowth_erdoes_turan}

With Theorem~\ref{thm:algrowth_mom-gen_logYn} at hand, we will first show the Erd\H{o}s-Tur\'an Law for $\log(Y_n)$. The complicated part is to transfer the result to $\log(O_n)$, see Lemma~\ref{lem:algrowth_closeness}, where we will need Theorem~\ref{thm:algrowth_total_variation}.

%
%

%
%
The proof of Theorem~\ref{thm:algrowth_CLT_logO_n} is a direct consequence of the following two lemmas. 
\begin{lemma}\label{lem:algrowth_CLT_logY_n}
Let $\hat{g}_{\Theta}$ be as in (\ref{eq:algrowth_g_Theta}) with $\gamma > 0$. Then we have, as $n \rightarrow \infty$, 
\begin{align*}
\frac{\log (Y_n) - G(n)}
{\sqrt{ F(n)}} 
\overset{d}{\longrightarrow} \mathcal{N}(0,1)
\end{align*}
where $\mathcal{N}(0,1)$ denotes the standard Gaussian distribution and $F(n)$ and $G(n)$ are as in Theorem~\ref{thm:algrowth_CLT_logO_n}.
\end{lemma}
\begin{remark}
Landau's result (\ref{eq:intro_asymp_Landau})
implies immediately that an analogous result to Lemma~\ref{lem:algrowth_CLT_logY_n} for $\log (O_n)$ can only be valid for $0 < \gamma < 1$. This means that $\log(Y_n)$ is a good approximation for $\log(O_n)$ when $0 < \gamma < 1$ but for $\gamma \geq 1$ the behavior of the two random variables is indeed different.
\end{remark}
\begin{proof}[Proof of Lemma~\ref{lem:algrowth_CLT_logY_n}]
 We write the the expansion in Theorem~\ref{thm:algrowth_mom-gen_logYn} as
\begin{align*}
\E_{\Theta}[\exp(s \log (Y_n))]
= \exp(f(n,s))
\end{align*}
and expand the function $f(n,s)$ around $s=0$. Now set $F(n)$ as in Lemma~\ref{lem:algrowth_CLT_logY_n}. Since the error terms in Theorem~\ref{thm:algrowth_mom-gen_logYn} are uniform in $s$ we can apply it for 
$s/ \sqrt{ F(n) }$. This gives, as $n \rightarrow \infty$,
\begin{align*}
\E_{\Theta}\Bigg[ \exp \bigg(s \frac{\log (Y_n)}{ \sqrt{F(n)}} \bigg)\Bigg] 
\sim 
\exp\bigg(\frac{s^2}{2} + \bar{G}(n) H(n) s\bigg) ,
\end{align*}
where $H(n)$ is defined as in Lemma~\ref{lem:algrowth_CLT_logY_n} and 
\begin{align*}
 \bar{G}(n) = \sqrt{\Gamma(\gamma) + \Gamma(1+\gamma)} \Big(\frac{n}{\Gamma(1+\gamma)}\Big)^{\frac{\gamma}{2(1+\gamma)}} \log \Big(\frac{n}{\Gamma(1+\gamma)}\Big).
\end{align*}
 By means of L\'{e}vy's continuity theorem the result follows.
\end{proof}

To transfer the result from $\log(Y_n)$ to $\log(O_n)$ we need to show that they are close in a certain sense. We will prove the following 

\begin{lemma}\label{lem:algrowth_closeness}
For $\theta_m = m^{\gamma}$ with $0<\gamma<1$ the following holds as $n \rightarrow \infty$:
\begin{align*}
 \mathbb{P}_{\Theta}\big(\log(Y_n)-\log(O_n) \geq \log(n)\log\log(n) \big) \rightarrow 0. 
\end{align*}
\end{lemma}

\begin{proof}
First, recall \eqref{eq:log_On_von_mangold} and notice that 
\begin{align*}
 \log(O_n) = \psi(n) - R(n)
\end{align*}
where
\begin{align*}
 \psi(n) = \sum_{k=1}^n \Lambda(k) \quad \text{and} \quad R(n) = \sum_{k=1}^n \Lambda(k)\one_{\{D_{nk}=0 \}}.
\end{align*}
Recall that $\psi$ is the so-called Chebyshev function which satisfies the asymptotic $\psi(n) = n(1+o(1))$ (where the error term has a positive sign). We need to identify the smallest $b$ such that for $g(n) = \log(n)\log\log(n)$
\begin{align}\label{eq:algrowth_closeness_first}
 \mathbb{P}_{\Theta}\Big(\log(Y_n)-  \psi(b) \geq \frac{g(n)}{2}\Big) \rightarrow 0. 
\end{align}
Lemma~\ref{lem:algrowth_CLT_logY_n} implies that
\begin{align*}
 \mathbb{P}_{\Theta}\Big(\log(Y_n)- h(n) \geq \epsilon \Big) \rightarrow 0
\end{align*}
for any $\epsilon > 0$ and functions $h$ such that $h(n) / n^{\frac{\gamma}{1+\gamma}}\log(n) \rightarrow \infty$. Therefore, choose $b = n^{\frac{\gamma}{1+\gamma}}\log^2(n)$, then (\ref{eq:algrowth_closeness_first}) is satisfied (actually, it holds for any positive function $g(n)$). It remains to prove 
\begin{align}\label{eq:algrowth_closeness_second}
 \mathbb{P}_{\Theta}\Big(R(n)-  \sum_{k=b+1}^n \Lambda(k) \geq \frac{g(n)}{2}\Big) \rightarrow 0. 
\end{align}
Notice that
\begin{align*}
R(n) - \sum_{k=b+1}^n \Lambda(k) 
\leq R(b)
\leq \sum_{k=1}^b \Lambda(k)\one_{\{C_{k}=0 \}}
=: S(b)
\end{align*}
and thus it suffices to show
\begin{align*}
 \mathbb{P}_{\Theta}\Big(S(b)\geq \frac{g(n)}{2}\Big) \rightarrow 0. 
\end{align*}
To prove this we will approximate $S(b)$ by the functional
\begin{align*}
 S'(b) := \sum_{k=1}^b \Lambda(k)\one_{\{Z_{k}=0 \}}
\end{align*}
 where the $Z_k$ are independent Poisson random variables with parameter $k^{\gamma -1} t^k $ as in \eqref{eq:tot_var_dist_d_bn}. Then
\begin{align*}
 \mathbb{P}_{\Theta}\Big(S(b)\geq \frac{g(n)}{2}\Big) = \mathbb{P}_{\Theta}^t\Big(S'(b)\geq \frac{g(n)}{2}\Big) + O\big(d_K(S(b),S'(b))\big),
\end{align*}
where $d_K(X,Y)$ denotes again the Kolmogorov distance of the random variables $X$ and $Y$. Clearly,
\begin{align*}
 d_K(S(b),S'(b)) \leq d_{TV}(S(b),S'(b)) \leq d_b(n),
\end{align*}
and Theorem~\ref{thm:algrowth_total_variation} shows that $d_b(n) \rightarrow 0$ if and only if $b = o(n^{\frac{1}{1+\gamma}})$. For $0 < \gamma < 1$ we have $b = n^{\frac{\gamma}{1+\gamma}}\log^2(n) = o(n^{\frac{1}{1+\gamma}})$. Therefore, it suffices to show
\begin{align*}
\mathbb{P}_{\Theta}^t\Big(S'(b)\geq \frac{g(n)}{2}\Big) \rightarrow 0,
\end{align*}
which is equivalent to
\begin{align}\label{eq:algrowth_closeness_third}
\log\mathbb{P}_{\Theta}^t\Big(e^{sS'(b)}\geq e^{\frac{s g(n)}{2}}\Big) \rightarrow -\infty,
\end{align}
for $s\geq 0$. The moment generating function of $S'(b)$ is given by
\begin{align*}
 \ETt{e^{sS'(b)}} = \prod_{k=1}^b \Big( 1+ e^{- k^{\gamma-1}t^k}\big(e^{s \Lambda(k)}-1\big) \Big),
\end{align*}
where $t^k = \exp(- k n^{-\frac{1}{1+\gamma}})$. By Markov's inequality and with $\log(1+z)\leq z$ we get
\begin{align*}
 \log\mathbb{P}_{\Theta}^t\Big(e^{s S'(b)}\geq e^{\frac{s g(n)}{2}}\Big) 
 &\leq  -\frac{s g(n)}{2} + \sum_{k=1}^b \log \Big( 1+ e^{- t^k k^{\gamma-1}}\big(e^{s \Lambda(k)}-1\big) \Big) \\
&\leq  -\frac{s g(n)}{2} +  \sum_{k=1}^b  e^{- t^k k^{\gamma-1}}\big(e^{s\Lambda(k)}-1\big)  \\
&\leq  -\frac{s g(n)}{2}+ \big(e^{s\log(n)}-1\big) \sum_{k=1}^b  e^{- t^k k^{\gamma-1}} .
\end{align*}
Since $b = o\big(n^{\frac{1}{1+\gamma}}\big)$, there is a constant $c>0$ such that
\begin{align*}
 \sum_{k=1}^b  \exp\big(-  t^k k^{\gamma-1}\big) 
&\leq \sum_{k=1}^b  \exp\big(-  k^{\gamma-1}\exp(- b n^{-\frac{1}{1+\gamma}})\big)\\
&= \sum_{k=1}^b  \exp\big(- c  k^{\gamma-1} )\\
&\leq \int_1^b \exp\big(- c  x^{\gamma-1} ) dx\\
&= O\Big(\Gamma\Big(\frac{1}{1-\gamma},b^{\gamma-1}\Big) - \Gamma\Big(\frac{1}{1-\gamma},1\Big)\Big)\\
&= O(1).
\end{align*}
Thus for $g(n) = \log(n)\log\log(n)$ and  $s:= \sqrt{\log\log(n)} / g(n)$ this yields
\begin{align*}
 \log\mathbb{P}_{\Theta}^t\Big(e^{s S'(b)}\geq e^{\frac{s}{2}}\Big) 
\leq -\frac{\sqrt{\log\log(n)}}{2} + O\big( \log\log(n)^{-1/2} \big).
\end{align*}
The proof is complete.
\end{proof}
%
%

\vskip 15pt

\subsection{Proof of Theorem~\ref{thm:fclt}}\label{subsection:FCLT}

In Remark~\ref{remark_bound_total_var_distance_too_small} was mentioned that the bound $b = o(n^{\frac{1}{1+\gamma}})$ is too small to study properties of the whole cycle count process via the independent approximating Poisson random variables. However, Theorem~\ref{thm:fclt} gives an example of how to exploit Theorem~\ref{thm:algrowth_total_variation} in order to study the behavior of the small cycles. 
Recall that for $x > 0$ we define $x^* := \lfloor x \, n^{\frac{\gamma}{1+\gamma}} \rfloor$ and 

\begin{align*}
B_n(x) := \frac{\log(O_{x^*}) -\frac{1}{1+\gamma} \, x^{\gamma} \log(n) \, n^{\frac{\gamma^2}{1+\gamma}}}{\sqrt{\frac{\gamma}{(1+\gamma)^2} \log^2(n) \, n^{\frac{\gamma^2}{1+\gamma}}}}.
\end{align*}
%
%
%
%
%
\begin{proof}[Proof of Theorem~\ref{thm:fclt}]
First notice that 
$$\log(Y_{x^*}) - \log(O_{x^*}) 
\leq \log(Y_n) - \log(O_n) $$
and thus by means of Lemma~\ref{lem:algrowth_closeness} it is sufficient to show that 
\begin{align*}
\frac{\log(Y_{x^*}) -\frac{1}{1+\gamma} \, x^{\gamma} \log(n) \, n^{\frac{\gamma^2}{1+\gamma}}}{\sqrt{\frac{\gamma}{(1+\gamma)^2} \log^2(n) \, n^{\frac{\gamma^2}{1+\gamma}}}}
\end{align*}
satisfies the required convergence.
Since $x^* = o(n^{\frac{1}{1+\gamma}})$ and in a discrete probability space $d_b(n) \rightarrow 0$ is equivalent to convergence in distribution of $(C_1, ..., C_b)$ to $(Z_1, ..., Z_b)$, Theorem ~\ref{thm:algrowth_total_variation} yields 
 \begin{align*}
 \ET{e^{s \log(Y_{x^*})}}
= \ETt{e^{s \log(P_{x^*})}} \big(1+o(1)\big)
\end{align*}
where $\log(P_{x^*}) = \sum_{m=1}^{x^*} \log(m)Z_m$.
%
%
%
%
%
%
Thus we have to show 
\begin{align*}
\frac{\log(P_{x^*}) -\frac{1}{1+\gamma} \, x^{\gamma} \log(n) \, n^{\frac{\gamma^2}{1+\gamma}}}{\sqrt{\frac{\gamma}{(1+\gamma)^2} \log^2(n) \, n^{\frac{\gamma^2}{1+\gamma}}}}
\overset{d}{\longrightarrow} \mathcal{W}(x^{\gamma}).
\end{align*}
The convergence of the finite dimensional distributions is easily established. By independence, the characteristic function of $\log(P_{x^*})$ is given by
\begin{align*}
\ETt{e^{is \log(P_{x^*})}} 
&= \exp \bigg( \sum_{m=1}^{x^*} m^{\gamma -1} t^m (e^{is \log(m)}-1) \bigg) \\
&= \exp \bigg( is \alpha(x^*) - \frac{s^2}{2} \beta(x^*) + \delta(s,x^*) \bigg)
\end{align*}
where
\begin{align*}
\alpha(x^*) = \sum_{m=1}^{x^*} m^{\gamma -1} t^m \log(m), 
\quad 
\beta(x^*) = \sum_{m=1}^{x^*} m^{\gamma -1} t^m \log^2(m), 
\end{align*}
and
\begin{align*}
\delta(s, x^*) = \sum_{j=3}^{\infty}\sum_{m=1}^{x^*} m^{\gamma -1} t^m \frac{(is)^j \log^{j}(m)}{j!}.
\end{align*}
With computations similar to those in the proof of Lemma~\ref{lem:mu_0b and others} we get
\begin{align*}
&\alpha(x^*) = \frac{1}{1+\gamma} \, x^{\gamma} \, n^{\frac{\gamma^2}{1+\gamma}} \, \log(n)  + O\big(n^{\frac{\gamma^2}{1+\gamma}}\big), \\
&\beta(x^*) = \frac{\gamma}{(1+\gamma)^2} x^{\gamma}\, n^{\frac{\gamma^2}{1+\gamma}} \, \log^2(n) + O\big(n^{\frac{\gamma^2}{1+\gamma}}\log(n)\big), \\
&\sum_{m=1}^{x^*} m^{\gamma -1} t^m \log^3(m) =  O\big(n^{\frac{\gamma^2}{1+\gamma}}\log(n)\big).
\end{align*}
This proves that for every fixed $x$ we have
\begin{align*}
\widetilde{B}_n(x) 
\overset{d}{\longrightarrow} \mathcal{N}(0,x^{\gamma}).
\end{align*}
It remains to prove that the process $\widetilde{B}_n(.)$ is tight. We use the moment condition given in \cite[Theorem 15.6]{Bi99}, that is we have to show that for any $n \geq 0$ and $0 \leq x_1 < x < x_2$
\begin{align*}
\mathcal{E}_{\Theta}(x_1, x_2):=
 \ET{(\widetilde{B}_n(x)-\widetilde{B}_n(x_1 ))^2(\widetilde{B}_n(x_2)-\widetilde{B}_n(x))^2}
= O\big((x_2-x_1)^2\big).
\end{align*}
To prove this we use the independence of the $Z_m$. Denote 
$$\log(P_{x^*}^{y^*}) = \sum_{m=x^*+1}^{y^*} \log(m)Z_m.$$
Then
\begin{align*}
\mathcal{E}_{\Theta}(x_1, x_2) 
&= O \bigg( \Big(x^{\gamma}\, n^{\frac{\gamma^2}{1+\gamma}} \, \log^2(n) \Big)^{-2}
\mathbb{V}_{\Theta}\big(\log(P_{x_1^*}^{x^*})\big) 
\mathbb{V}_{\Theta}\big(\log(P_{x^*}^{x_2^*}) \big)\bigg) \\
&= O \bigg( \Big(x^{\gamma}\, n^{\frac{\gamma^2}{1+\gamma}} \, \log^2(n) \Big)^{-2}
\big(\beta(x^*)-\beta(x_1^*)\big) 
\big(\beta(x_2^*)-\beta(x^*)\big)\bigg)\\
&
= O \big( (x^{\gamma}-x_1^{\gamma})(x_2^{\gamma}-x^{\gamma})\big)
= O\big((x_2-x_1)^2\big).
\end{align*}
This completes the proof.
\end{proof}


\subsection*{Acknowledgments}
The research leading to these results has been supported by the SFB $701$
(Bielefeld) and has received funding from the People Programme (Marie
Curie Actions) of the European Union's Seventh Framework Programme
(FP7/$2007-2013$) under REA grant agreement nr.$291734$.

\bibliographystyle{acm}
\bibliography{literatur}

\end{document}